\documentclass{article}
\usepackage[utf8]{inputenc}
\usepackage{amsmath}
\usepackage{amsfonts}
\usepackage{amssymb}
\usepackage{tikz-cd}
\usepackage{color,soul}
\usepackage{array}
\usepackage{zref-savepos}
\usepackage{amsthm}
\usepackage{multirow, bigstrut}
\usepackage[font={small,it}]{caption}
\usepackage{tikz}
\usepackage{bm}
\usepackage{titling}
\usetikzlibrary{arrows}

\usepackage{hyperref}
\hypersetup{
    colorlinks=true,
    linkcolor=blue,
    filecolor=magenta,  
    urlcolor=cyan,
}\usepackage[capitalize, nameinlink]{cleveref}
\usepackage{thmtools, thm-restate}

\usetikzlibrary{calc}
\usepackage{comment}
\usepackage{enumerate}
\usepackage{tikz}
\usetikzlibrary{shapes.geometric}

\newtheorem{theorem}{Theorem}[section]
\newtheorem*{theorem*}{Theorem}
\newtheorem{lemma}[theorem]{Lemma}
\newtheorem{proposition}[theorem]{Proposition}
\newtheorem{corollary}[theorem]{Corollary}

\theoremstyle{definition}
\newtheorem{definition}[theorem]{Definition}

\newtheorem{example}[theorem]{Example}
\newtheorem{question}[theorem]{Question}

\theoremstyle{plain}

\newcommand{\C}{\mathbb{C}}
\newcommand{\R}{\mathbb{R}}
\newcommand{\N}{\mathbb{N}}
\newcommand{\Z}{\mathbb{Z}}

\DeclareMathOperator{\tr}{tr}

\DeclareMathOperator{\cpc}{Cap}

\DeclareMathOperator{\per}{per}

\DeclareMathOperator{\Mat}{Mat}

\DeclareMathOperator{\HStab}{HStab}
\DeclareMathOperator{\SLC}{SLC}

\DeclareMathOperator{\MD}{MD}

\DeclareMathOperator{\MatTup}{MatTup}
\DeclareMathOperator{\PSD}{PSD}
\DeclareMathOperator{\conv}{conv}

\usepackage[backend=bibtex, style=alphabetic, backref=true,maxbibnames=99, url=false]{biblatex} 
\title{Unique Minimizers for Permanents, Mixed Discriminants, and Log-concave Polynomials}

\author{Leonid Gurvits and Jonathan Leake}

\begin{document}

\sloppy

\maketitle

\begin{abstract}
    The permanent and mixed discriminant of positive matrices are classic problems for which we do not expect an efficient algorithm for exact computation. Thus much work has been done to understand how well we can bound and approximately compute these quantities. One line of research in this area begins with the results of \cite{Gurvits2008} where van der Waerden lower bounds of $\frac{n!}{n^n}$ are proven for doubly stochastic inputs for both problems, using a simple proof via stable polynomials. Along with the bound itself, the same techniques are used to show that the permanent and mixed discriminant are uniquely minimized at a certain natural symmetric input.

    In this paper, we generalize those results in two ways. First, we extend the unique minimization results beyond doubly stochastic inputs to other marginals which are near doubly stochastic. This yields the first such unique minimization results for the mixed discriminant beyond the doubly stochastic case. We also discuss why one cannot hope similar results to hold in general for all marginals. Second, we extend the unique minimization result for real stable polynomials to strongly log-concave (aka Lorentzian) polynomials in the doubly stochastic case. This captures an analogous previous result on unique minimization for the mixed volume. Finally, we discuss various open problems related to these results.
\end{abstract}

\section{Introduction}

The permanent has received significant attention due to its connections to fundamental problems in combinatorics and computer science. This has led to various questions and answers regarding bounding and approximating the permanent under a wide range of different assumptions. One of the most famous is that of the Van der Waerden conjecture from the 1930s (settled in 1980ish by \cite{Fal81,Egor81}), which says
\[
    \per(M) \geq \frac{n!}{n^n}
\]
for all doubly stochastic $M \in \R_{\geq 0}^{n \times n}$ (i.e., $M\bm{1} = M^\top\bm{1} = \bm{1}$), and that $M = \frac{1}{n} \bm{1} \cdot \bm{1}^\top$ is the unique minimizer over this set of matrices.

This fact was, for example, used to give the first deterministic $e^n$-approximation algorithm for the permanent of a matrix with non-negative entries \cite{LSW98}. Note that computing the permanent is \#P-complete and thus efficient exact computation is not expected to be possible. This is true even when $M \in \{0,1\}^{n \times n}$, where $\per(M)$ counts perfect matchings of the associated bipartite graph.

Around 20 years ago, the first author initiated a line of work which gave simple proofs of these facts using real stable polynomials and polynomial capacity \cite{Gur07VdW,Gurvits2008}. This proof method has the added benefit of also applying to the mixed discriminant of positive semidefinite matrices \cite{Gur06MD}, and to some extent to the mixed volume of convex bodies \cite{Gur09MV}. Later works extended these techniques to obtain approximation algorithms for various other combinatorial and probabilistic quantities as well, including the traveling salesman problem \cite{KKO21,GKL24}. These works generalized the doubly stochastic lower bound and rank-one unique minimization result for the permanent, to matrices and real stable polynomials more generally which are only ``close'' to being doubly stochastic. This extra robustness in the bounds proved to be a crucial step towards these new applications.

The goal of this paper is then to extend the robust rank-one minimizer uniqueness results to new regimes. Specifically, we $(1)$ generalize recent unique minimization results for the permanent to the mixed discriminant, $(2)$ generalize coefficient minimization uniqueness for real stable polynomials to strongly log-concave polynomials, and $(3)$ provide some counterexamples to possible extensions of these and related results.

In the rest of this section, we state and discuss our main results. See \Cref{sec:prelims} for any not-yet-defined notions.

\subsection{Main Result: Mixed Discriminant Unique Rank-One Minimizer} \label{sec:main-MD}

We now extend the fact that $\frac{1}{n} \bm{1}\bm{1}^\top$ uniquely minimizes the permanent over all doubly stochastic matrices in a variety of ways. We first recall the definition of the \textbf{permanent}:
%
% We first give our main results on unique minimizers of certain symmetric multilinear forms. The first relates to the \textbf{permanent}, defined via
\[
    \per(M) = \sum_{\sigma \in S_n} \prod_{i=1}^n m_{i,\sigma(i)}
\]
for a given $n \times n$ matrix $M$. The following is an extension of the unique minimization result for the permanent proven in \cite{GKL24}.

\begin{theorem}[\cite{GKL24}] \label{main-per-unique-min}
    For all $n$ there exists $\epsilon > 0$ such that: for all $\bm{c}$ such that $\|\bm{c}-\bm{1}\|_1 < \epsilon$ and $\sum_{i=1}^n c_i = n$, the permanent is uniquely minimized by $M = \frac{1}{n} \bm{1} \cdot \bm{c}^\top$ over the set of non-negative entry $n \times n$ matrices with row sums $\bm{1}$ and column sums $\bm{c}$.
\end{theorem}

Our first main result generalizes this to the \textbf{mixed discriminant}, and hence also reproves \Cref{main-per-unique-min}. Recall its definition:
\[
    \MD(X_1,\ldots,X_n) = \frac{1}{n!} \partial_{t_1} \cdots \partial_{t_n} \det\big(t_1 X_1 + \cdots + t_n X_n\big)
\]
for a given tuple of $n \times n$ matrices $(X_1,\ldots,X_n)$.

\begin{theorem} \label{main-MD-unique-min}
    For all $n$ there exists $\epsilon > 0$ such that: for all $C$ of trace $n$ and eigenvalues $\bm{c}$ such that $\|\bm{c} - \bm{1}\|_1 < \epsilon$, the mixed discriminant is uniquely minimized by $(X_1,\ldots,X_n) = \frac{1}{n}(C,\ldots,C)$ over $n$-tuples $(X_1,\ldots,X_n)$ of trace-$1$ PSD matrices such that $\sum_{i=1}^n X_i = C$.
\end{theorem}

The proof of this result can be found in \Cref{sec:proof-MD-rank-one-minimizer}, where we also give more explicit bounds on the value of $\epsilon$ required for the result.

We finally note that unique minimization of the mixed discriminant in the doubly stochastic case (i.e. when $\|\bm{c}-\bm{1}\|_1 = 0$) was proven in \cite{Gur06MD} using the known equality conditions of the Alexandrov-Fenchel inequalities for the mixed discriminant. We are not aware of such easily understood and verified conditions in the strongly log-concave case, even for mixed volumes; the proofs in this paper do not use the Alexandrov-Fenchel inequalities whatsoever. In a certain sense, the proof path using Alexandrov-Fenchel is hidden in the inductive proofs of analogues of the Van der Waerden and Schrijver's inequalities for stable and strongly-loncave polynomials, where the equality conditions are replaced by much more elementary convex optimization arguments.

% using the Alexandrov-Fenchel inequalities. One further contribution of this paper is that we do not need the Alexandrov-Fenchel inequalities to prove \Cref{main-MD-unique-min}.

\subsection{Main Result: Strongly Log-concave Unique Rank-one Minimizer}

The permanent result and the mixed discriminant result are both obtained by analyzing real stable polynomials. See \Cref{sec:lc-polys} for definitions and discussion of classes of log-concave polynomials. One general result in this direction is given as follows.

\begin{theorem}[\cite{Gur07VdW}] \label{rs-unique-min}
    For all $n$, the all-ones coefficient is uniquely minimized by $p(\bm{x}) = \left(\frac{x_1 + \cdots + x_n}{n}\right)^n$ over homogeneous real stable polynomials $p$ such that $p(\bm{1}) = 1$ and $\nabla p(\bm{1}) = \bm{1}$.
\end{theorem}

Our second result is to generalize \Cref{rs-unique-min} by extending it beyond real stable to strongly log-concave polynomials \cite{Gur09} (also called Lorentzian \cite{BH20} and completely log-concave \cite{ALOGV18iii}). Previous results utilized specific properties of real stable polynomials, and here we are able to generalize beyond this setting.

\begin{theorem} \label{main-SLC-unique-min}
    For all $n$, the all-ones coefficient of $p$ is uniquely minimized by $p(\bm{x}) = \left(\frac{x_1 + \cdots + x_n}{n}\right)^n$ over homogeneous strongly log-concave polynomials $p$ such that $p(\bm{1}) = 1$ and $\nabla p(\bm{1}) = \bm{1}$.
\end{theorem}

We actually prove a more general version of this result based on polynomial capacity; see \Cref{sec:proof-SLC-rank-one-minimizer} for the statement and proof.

As a last comment, \Cref{main-SLC-unique-min} was proven for volume polynomials (which are strongly log-concave) in \cite{Gur09MV} using a more complicated proof. The main difference in the proofs comes from the fact that we are able to show an inductive property of doubly stochastic minimizers; i.e. that the partial derivatives must also be doubly stochastic. See Step 3 of \Cref{subsec:proof-SLC-unique-min} for further details. See also the last paragraph of \Cref{sec:main-MD} for discussion on the connection of these strongly log-concave results to the Alexandrov-Fenchel inequalities.

\subsection{Main Result: Failure of Rank-One Minimizers}

To understand the delicateness of the above results, we now provide some negative results for statements which naturally generalize the above results. The first result shows why one of the marginal vectors needs to be $\bm{1}$ in the above results.

\begin{proposition} \label{prop:diff-r-c}
    For all $n$ and all $\bm{r},\bm{c} \neq \bm{1}$, the rank-one matrix $\frac{1}{n} \bm{r} \bm{c}^\top$ does not minimize the permanent over the set of non-negative entry $n \times n$ matrices with row sums $\bm{r}$ and column sums $\bm{c}$.
\end{proposition}

And finally, even in the case where the column sums are $\bm{1}$, the row sums need to be close to $\bm{1}$ for rank-one unique minimization results to hold. We note that this result already appears in the conference paper \cite{GKL24}, but we include it and its proof here also for completeness.

\begin{proposition}[\cite{GKL24}] \label{prop:boundary-minimizer}
    For all $n$, there exists $\bm{c}$ such that: every non-negative entry $n \times n$ matrix with row sums $\bm{1}$ and column sums $\bm{c}$ has positive permanent, but the set of such matrices is minimized on its relative boundary.
\end{proposition}

The proofs of these results can be found in \Cref{sec:proofs-failure-results}.

\subsection{Open Questions} \label{sec:open}

A number of open questions remain. First of all, there is another version of the rank-one unique minimizer question for the mixed discriminant.

\begin{question}
    For all $n$, does there exist $\epsilon > 0$ such that: for all $\bm{t}$ such that $\|\bm{t}-\bm{1}\|_1 < \epsilon$ and $\sum_{i=1}^n t_i = n$, the mixed discriminant is uniquely minimized by $(X_1,\ldots,X_n) = \frac{1}{n} (t_1 I, \ldots, t_n I)$ over $n$-tuples $(X_1,\ldots,X_n)$ of PSD matrices such that $\tr(X_i) = t_i$ and $\sum_{i=1}^n X_i = I$?
\end{question}

This question is natural because it can generalized to all real stable polynomials, as given in the following more general question.

\begin{question}
    For all $n$, does there exist $\epsilon > 0$ such that: for all $\bm{t}$ such that $\|\bm{t}-\bm{1}\|_1 < \epsilon$ and $\sum_{i=1}^n t_i = n$, the coefficient of $x_1x_2 \cdots x_n$ is uniquely minimized by $p(\bm{x}) = \frac{1}{n^n}\left(\sum_{i=1}^n t_i x_i\right)^n$ over $n$-homogeneous $n$-variate real stable polynomials $p$ such that $p(\bm{1}) = 1$ and $\nabla p(\bm{1}) = \bm{t}$?
\end{question}

Another natural open question here is related to the tightness of capacity bounds on a certain inner product applied to determinantal polynomials.

\begin{question} \label{QP}
    Fix $n \leq N$, fix unit vectors $\bm{u}_1,\ldots,\bm{u}_N,\bm{v}_1,\ldots,\bm{v}_N \in \C^n$, and fix $0 < a_1,\ldots, a_N \leq 1$ such that $\sum_{i=1}^N a_i \bm{u}_i\bm{u}_i^* = \sum_{i=1}^N a_i \bm{v}_i\bm{v}_i^* = I$ and $\sum_{i=1}^N a_i = n$. With this, define
    \[
        p(\bm{x}) = \det\left(\sum_{i=1}^N x_i \sqrt{a_i} \bm{u}_i\bm{u}_i^*\right), \quad q(\bm{x}) = \det\left(\sum_{i=1}^N x_i \sqrt{a_i} \bm{v}_i\bm{v}_i^*\right).
    \]
    Do we always have that $\sum_{\bm\alpha} p_{\bm\alpha} q_{\bm\alpha} \geq \frac{n!}{n^n}$, where $p_{\bm\alpha}$ is the $\bm{x}^{\bm\alpha}$ coefficient of $p(\bm{x})$?
\end{question}

An affirmative answer to this question is actually equivalent to Conjecture~6.1 from \cite{Gur06MD} on the quantum permanent (also known as the Cayley hyperdeterminant, or $4$-dimensional Pascal determinant) of doubly stochastic separable PSD block matrices. The left-hand side of the inequality in \Cref{QP} is precisely the quantum permanent of the associated matrix. A weaker version of the bound (slightly better than $e^{-n}$, see \Cref{sec:aside-QP}) in \Cref{QP} can be derived from the results of \cite{AOG17}, and a more general version of the conjecture also appears as Remark~3.10 in \cite{Gur09}. We discuss all of this further in \Cref{sec:aside-QP}.

We further note that the quantum permanent of a general doubly stochastic PSD block matrix is known to not satisfy such a bound; in fact, it can even take the value $0$. That said, efficiently algorithmically determining precisely when it is $0$ is an open problem \cite{GargGurvitsOliveiraWigderson2020,Gurvits2004} whose resolution would yield efficient algorithms for a version of the famous PIT (polynomial identity testing) problem.

Finally, there is quite a bit of literature (e.g. \cite{Foregger1980,Minc1984,Brualdi1985,Hwang1985,ForeggerSinkhorn1986,SongHongJunKim1997,PulaSongWanless2011,SongBeasley2024} etc.) on minimizers of the permanent over faces of the Birkhoff polytope. For example, it is open for many classes of faces of the Birkhoff polytope whether or not the permanent is minimized in its relative interior (such faces are called \textbf{cohesive} in the literature; see the survey \cite{CheonWanless2005}). We briefly tried to apply techniques from this paper to various classes of such faces, but we could not immediately obtain anything new. A natural question is then whether or not this is possible in some cases, and how such statements would generalize to matrices with column sums $\bm{1}$ and row sums $\approx \bm{1}$.

\subsection{AI Declaration}

We note that the literature review for the last paragraph of \Cref{sec:open}, along with some testing to see if our results could apply in the context of cohesive faces of the Birkhoff polytope, was completed with the help of AI (ChatGPT 5.6 Sol). The rest of this paper was created by the authors without the use of AI, except for creating some .bib file entries and some basic Google queries.

\section{Preliminaries} \label{sec:prelims}

Throughout we use standard vector and set notation. We let $\C,\R,\Z$ denote the complex, real, and integer numbers, respectively. We use a subscript to modify these sets; for example $\Z_{\geq 0}$ denote the non-negative integers. We also use other sensible notation; e.g. we let $\R^{n \times n}$ denote the set of $n \times n$ matrices with reel entries. Also, given $\bm{x},\bm{y} \in \R^n$ we use $\bm{x} \leq \bm{y}$ to denote $x_i \leq y_i$ for all $i \in [n]$, we denote $\bm{x} \odot \bm{y} = (x_1y_1,\ldots,x_ny_n)$, we denote $\bm{x}^{\bm{y}} = x_1^{y_1} \cdots x_n^{y_n}$ (whenever $\bm{x} > \bm{0}$), we denote $\|\bm{x}\|_1 = \sum_{i=1}^n |x_i|$, and we denote $\langle \bm{x}, \bm{y} \rangle = \sum_{i=1}^n x_i y_i$.

We will make use of a number of shorthands for various spaces throughout. Specifically, we define certain matrix spaces:
\begin{itemize}
    \item $\Mat_n(\bm{c}) = \left\{M \in \R_{\geq 0}^{n \times n} : M\bm{1} = \bm{1}, ~ M^\top\bm{1} = \bm{c}\right\}$,
    \item $\PSD_n = $ the set of $n \times n$ positive semidefinite matrices,
    \item $\MatTup_n(C) = \left\{(M_1,\ldots,M_n) \in \PSD_n^n : \tr(M_i) = 1 ~ \forall i, ~ \sum_{i=1}^n M_i = C\right\}$.
    % \item $\HStab_n(\bm{c}) = \left\{p \in \R_{\geq 0}[x_1,\ldots,x_n]\right\}$
\end{itemize}

\subsection{Log-concave Polynomials} \label{sec:lc-polys}

Here we give some definitions and useful notation, before recalling a few results we will make use of.

\begin{definition}
    An $n$-homogeneous polynomial $p \in \R_{\geq 0}[x_1,\ldots,x_n]$ is \textbf{real stable} if $p(z_1,\ldots,z_n) \neq 0$ for all $z_1,\ldots,z_n$ in the (open) complex upper half-plane.
\end{definition}

\begin{definition}
    An $n$-homogeneous polynomial $p \in \R_{\geq 0}[x_1,\ldots,x_n]$ is \textbf{strongly log-concave (SLC)} if $\partial_{x_1}^{\alpha_1} \cdots \partial_{x_n}^{\alpha_n} p(\bm{x})$ is either log-concave on $\R_{>0}^n$ or identically zero for all $\bm\alpha \in \Z_{\geq 0}^n$.
\end{definition}

\begin{example}
    For all $A_1,\ldots,A_n \in \PSD_d$, the polynomial $p(\bm{x}) = \det\left(\sum_{i=1}^n x_i A_i\right)$ is real stable. By considering diagonal matrices, this implies products of linear forms with non-negative coefficients are also real stable. Finally, for all compact convex bodies $K_1,\ldots,K_n \subseteq \R^d$, the polynomial $p(\bm{x}) = \mathrm{vol}\left(\sum_{i=1}^n x_i K_i\right)$ is strongly log-concave, but not necessarily real stable.
\end{example}

We will not explicitly use the following, but it is a well-known fact in the literature.

\begin{proposition}
    Every homogeneous real stable polynomial is strongly log-concave.
\end{proposition}

We note that strongly log-concave is equivalent to \textbf{Lorentzian} \cite{BH20} and \textbf{completely log-concave} \cite{ALOGV18iii} for homogeneous polynomials. That said, we also give one other equivalent definition of strongly log-concave which relies upon the following folklore result in convex analysis.

\begin{lemma}
    If $p \in \R[x_1,\ldots,x_n]$ is $d$-homogeneous and $\bm{x} \in \R^n$ is such that $p(\bm{x}) > 0$, then the Hessian of $\log p$ at $\bm{x}$ is negative semidefinite if and only if the Hessian of $p^{1/d}$ at $\bm{x}$ is negative semidefinite.
\end{lemma}

This leads to the following alternative definition of strongly log-concave polynomials.

\begin{corollary} \label{alt-slc-def-1/d}
    If $p \in \R_{\geq 0}[x_1,\ldots,x_n]$ is $d$-homogeneous, then $p$ is log-concave on $\R_{>0}^n$ if and only if $p^{1/d}$ is concave on $\R_{>0}^n$.
\end{corollary}

We now state a few basic properties of real stable and SLC polynomials from the literature, which have straightforward proofs. We first need a definition, which describes the Newton's inequalities that hold for real-rooted polynomials.

\begin{definition} \label{def:n-Newton}
    A polynomial $f \in \R_{\geq 0}[t]$ of degree at most $n$ is \textbf{$n$-Newton} if for $f(t) = \sum_{k=0}^n \binom{n}{k} f_k t^k$ we have $f_k^2 \geq f_{k-1} f_{k+1}$ for all $k$.
\end{definition}

\begin{lemma}
    Both the classes of real stable and strongly log-concave polynomials are preserved under taking partial derivatives and under external fields (i.e. $p(\bm{x}) \mapsto \frac{p(\bm{a} \odot \bm{x})}{p(\bm{a})}$ for $\bm{a} \in \R_{>0}^n$).
\end{lemma}

\begin{lemma} \label{slc-univariate-eval}
    Let $p \in \R_{\geq 0}[x_1,\ldots,x_n]$ be $d$-homogeneous, and fix $x_1,\ldots,x_{n-1} \in \R_{\geq 0}$. If $p$ is real stable then $f(t) = p(x_1,\ldots,x_{n-1},t)$ is real-rooted, and if $p$ is strongly log-concave then $f(t) = p(x_1,\ldots,x_{n-1},t)$ is $d$-Newton.
\end{lemma}

Finally, we generalize some of the notation use for matrices above to polynomials. Specifically we define the following spaces of polynomials:
\begin{itemize}
    \item $\HStab_n(\bm{c})$ is the set of all $n$-variate $n$-homogeneous real stable polynomials $p$ such that $p(\bm{1}) = 1$ and $\nabla p(\bm{1}) = \bm{c}$, and
    \item $\SLC_n(\bm{c})$ is the set of all $n$-variate $n$-homogeneous strongly log-concave polynomials $p$ such that $p(\bm{1}) = 1$ and $\nabla p(\bm{1}) = \bm{c}$.
\end{itemize}
We will also call a homogeneous polynomial $p$ \textbf{doubly stochastic} if $p(\bm{1}) = 1$ and $\nabla p(\bm{1}) = \bm{1}$. Note that this implies $p$ is $n$-variate and $n$-homogeneous.

\subsection{Polynomial Capacity Bounds}

In this section, we recall various results from the literature on polynomial capacity. Given $p \in \R_{\geq 0}[x_1,\ldots,x_n]$ and $\bm\alpha \in \R_{\geq 0}^n$, the \textbf{$\bm\alpha$-capacity} of $p$ is defined as
\[
    \cpc_{\bm\alpha}(p) := \inf_{\bm{x} \in \R_{>0}^n} \frac{p(\bm{x})}{\bm{x}^{\bm\alpha}}.
\]
We often just say \textbf{capacity} when $\bm\alpha$ is understood from context. Polynomial capacity has been used in various contexts to lower-bound combinatorial and probabilistic quantities which are hard to compute. We now recall some of these bounds, which quantifies the difference between the capacity of a polynomial and its coefficients.

\begin{theorem}[\cite{Gur09}] \label{thm:Gurvits-SLC}
    Let $p \in \R_{\geq 0}[x_1,\ldots,x_n]$ be a $d$-homogeneous strongly log-concave polynomial. Fix $\bm\alpha \in \R_{\geq 0}^n$ such that $\sum_i \alpha_i = d$. Then
    \[
        \frac{p_{\bm\alpha}}{\cpc_{\bm\alpha}(p)} \geq \binom{d}{\bm\alpha} \frac{\alpha_1^{\alpha_1} \cdots \alpha_n^{\alpha_n}}{d^d},
    \]
    where $p_{\bm\alpha}$ is the coefficient of $\bm{x}^{\bm\alpha}$ in $p$.
\end{theorem}

\begin{theorem}[{\cite[Thm. 5.1]{Gur15}}] \label{thm:Gurvits}
    Let $p \in \R_{\geq 0}[x_1,\ldots,x_n]$ be a $d$-homogeneous real stable polynomial. Fix $\bm\alpha \in \R_{\geq 0}^n$ such that $\sum_i \alpha_i = d$, and let $d_i$ be the degree of $x_i$ in
    \[
        \left.\partial_{x_{i+1}}^{\alpha_{i+1}} \cdots \partial_{x_n}^{\alpha_n} p\right|_{x_{i+1} = \cdots = x_n = 0}
    \]
    for $i = 1,\ldots,n$. Then we have
    \[
        \frac{p_{\bm\alpha}}{\cpc_{\bm\alpha}(p)} \geq \prod_{i=2}^n \binom{d_i}{\alpha_i} \frac{\alpha_i^{\alpha_i} (d_i-\alpha_i)^{d_i-\alpha_i}}{d_i^{d_i}},
    \]
    where $p_{\bm\alpha}$ is the coefficient of $\bm{x}^{\bm\alpha}$ in $p$.
\end{theorem}

It is straightforward to see that the lower bound of \Cref{thm:Gurvits} is strictly decreasing in $d_i$ for fixed $\alpha_i$. With this, we obtain the following immediate corollary of \Cref{thm:Gurvits}.

\begin{corollary} \label{boundary-cap-bound}
    Let $p \in \R_{\geq 0}[x_1,\ldots,x_n]$ be an $n$-homogeneous real stable polynomial such that $p(\bm{e}_i) = 0$ for some $i \in [n]$. Then
    \[
        \frac{\partial_{x_1} \cdots \partial_{x_n} p(\bm{x})}{\cpc_{\bm{1}}(p)} \geq \frac{(n-2)^{n-2}(n-1)!}{(n-1)^{2n-3}} > \frac{n!}{n^n}.
    \]
\end{corollary}

% The next result is a specific corollary of \Cref{thm:Gurvits} which we will use in proving our main uniqueness of minimizers results.

Thee next result now gives a lower bound on the capacity itself for real stable polynomials under certain conditions on the gradient of the polynomial. This and similar results were originally used to prove \Cref{main-per-unique-min} and to improve the approximation factor for the traveling salesman problem.

\begin{theorem}[\cite{GL21}] \label{thm:rs-cap-bound}
    If $p \in \R[x_1,\ldots,x_n]$ is a $n$-homogeneous real stable polynomial such that $p(\bm{1}) = 1$ and $\|\bm{1}-\nabla p(\bm{1})\|_1 < 2$, then
    \[
        \cpc_{\bm{1}}(p) \geq \left(1 - \frac{\|\bm{1}-\nabla p(\bm{1})\|_1}{2}\right)^n.
    \]
\end{theorem}

% \begin{theorem}[\cite{GKL24}] \label{thm:rs-cap-bound-kappa}
%     If $\bm\kappa \in \Z_{\geq 0}^n$ and $p \in \R[x_1,\ldots,x_n]$ is a not necessarily homogeneous real stable polynomial such that $p(\bm{1}) = 1$ and $\|\bm\kappa-\nabla p(\bm{1})\|_1 < 1$, then
%     \[
%         \cpc_{\bm\kappa}(p) \geq \left(1 - \|\bm{1}-\nabla p(\bm{1})\|_1\right)^n.
%     \]
% \end{theorem}

Finally, we give one final result on the capacity of product of linear forms. This result is known in the literature, but we give a short proof here. We note that this sort of result can also be proven for capacity on completely positive operators.

\begin{lemma} \label{cap-bound-transpose}
    For any $n \times n$ matrix $M$ with non-negative entries, we have
    \[
        \cpc_{\bm{1}}\left(\prod_{i=1}^n (M\bm{x})_i\right) = \frac{1}{n^n} \left(\inf_{\bm{x},\bm{y} > \bm{0}} \frac{\langle M\bm{x},\bm{y}\rangle}{\bm{x}^{\bm{1}} \bm{y}^{\bm{1}}}\right)^n.
    \]
    In particular,
    \[
        \cpc_{\bm{1}}\left(\prod_{i=1}^n (M\bm{x})_i\right) = \cpc_{\bm{1}}\left(\prod_{i=1}^n (M^\top\bm{x})_i\right).
    \]
\end{lemma}
\begin{proof}
    Let $p_M(\bm{x}) = \prod_{i=1}^n (M\bm{x})_i$. For all $\bm{x},\bm{y} > \bm{0}$, by AM-GM we have
    \[
        \left(\frac{\langle \bm{x},\bm{y} \rangle}{n}\right)^n \geq \prod_{i=1}^n x_iy_i = \bm{x}^{\bm{1}} \bm{y}^{\bm{1}},
    \]
    with equality if and only if $y_i = x_i^{-1}$ for all $i \in [n]$. Thus we first have
    \[
        \frac{1}{n^n} \left(\inf_{\bm{x},\bm{y} > \bm{0}} \frac{\langle M\bm{x},\bm{y}\rangle}{\bm{x}^{\bm{1}} \bm{y}^{\bm{1}}}\right)^n \geq \inf_{\bm{x} > \bm{0}} \frac{\prod_{i=1}^n (M\bm{x})_i}{\bm{x}^{\bm{1}}} = \cpc_{\bm{1}}(p_M).
    \]
    By the above equality condition, we also have
    \[
        \frac{1}{n^n} \left(\inf_{\bm{x},\bm{y} > \bm{0}} \frac{\langle M\bm{x},\bm{y}\rangle}{\bm{x}^{\bm{1}} \bm{y}^{\bm{1}}}\right)^n \leq \frac{1}{n^n} \left(\inf_{\bm{x} > \bm{0}} \frac{\langle M\bm{x},(M\bm{x})^{-1}\rangle}{\bm{x}^{\bm{1}} (M\bm{x})^{-\bm{1}}}\right)^n = \cpc_{\bm{1}}(p_M)
    \]
    which completes the proof.
\end{proof}

% \begin{lemma} \label{cap-bound-transpose}
%     For any $n \times n$ matrix $M$ with non-negative entries, we have
%     \[
%         \cpc_{\bm{1}}\left(\prod_{i=1}^n (M\bm{x})_i\right) = \sup_{A \in \Mat_n(\bm{1})} \prod_{i,j=1}^n \left(\frac{m_{ij}}{a_{ij}}\right)^{a_{ij}}.
%     \]
%     In particular,
%     \[
%         \cpc_{\bm{1}}\left(\prod_{i=1}^n (M\bm{x})_i\right) = \cpc_{\bm{1}}\left(\prod_{i=1}^n (M^\top\bm{x})_i\right).
%     \]
% \end{lemma}
% \begin{proof}
%     Let $p_i(\bm{x}) = (M\bm{x})_i$, and define $f_i(\bm{y}) = \log p_i(e^{\bm{y}})$. Then $-\log \cpc_{\bm{1}}(p) = \left(f_1 + \cdots + f_n\right)^*(\bm{1})$, where $f^*$ denotes the convex conjugate of $f$:
%     \[
%         f^*(\bm\alpha) = \sup_{\bm{y} \in \R^n} \left[\langle \bm{y}, \bm{a} \rangle - f(\bm{y})\right].
%     \]
%     Since infimal convolution commutes with addition under the convex conjugation (e.g. \cite[Thm. 16.4]{Rockafellar1970Convex}), we have
%     \[
%         \left(f_1 + \cdots + f_n\right)^*(\bm{1}) = \inf_{\sum_i \bm{\alpha}_i = \bm{1}}\left(f_1^*(\bm\alpha_1) + \cdots + f_n^*(\bm\alpha_n)\right).
%     \]
%     A straightforward calculus computation gives
%     \[
%         -f_i^*(\bm\alpha_i) = \log\cpc_{\bm{\alpha_i}}(p_i) = \begin{cases}
%             \sum_{j=1}^n \alpha_{ij} \log\left(\frac{m_{ij}}{\alpha_{ij}}\right), & \sum_{j=1}^n \alpha_{ij} = 1 \\
%             -\infty, & \text{otherwise.}
%         \end{cases}
%     \]
%     Combining these facts gives the desired result.
% \end{proof}

\section{Optimizing Forms}

In this section, we give some general results on optimizing forms. These results are mainly about determining what sort of optima a form can have on a given feasible region. The main result of this section is \Cref{thm:diagonal-minimizers}, which proves that minimizers of symmetric multilinear forms over symmetric compact convex sets are symmetric, so long as no minimizers occur on the boundary. We will later use this result to prove our main result on the uniqueness of the rank-one minimizer for the mixed discriminant. That said, we have decided to state the results here more generally, as they seem interesting in their own right.

The first result is on quadratic forms which we will bootstrap to prove the main result.

\begin{lemma} \label{lem:at-most-one-min}
    Let $q(\bm{x})$ be a quadratic form on $\R^n$, and let $S \subset \R^n$ be a closed convex set containing no affine lines. If $\min\{q(\bm{x}) : \bm{x} \in S\}$ has no minimizers on the relative boundary of $S$, then it has at most one minimizer in $S$.
\end{lemma}
\begin{proof}
    Suppose $\min\{q(\bm{x}) : \bm{x} \in S\}$ has no minima on the relative boundary $S$, but $\bm{u} \neq \bm{v}$ are two minima in the relative interior of $S$. Letting $W \subseteq \R^n$ denote the orthogonal complement to the vector space parallel to the affine span of $S$, minimality in the relative interior implies $\nabla q(\bm{u}), \nabla q(\bm{v}) \in W$. Since $q$ is a quadratic form, its gradient is a linear form, and thus $\nabla q(\bm{y}) \in W$ for all $\bm{y}$ in the line $L$ through $\bm{u}$ and $\bm{v}$. Therefore every $\bm{y} \in S \cap L$ minimizes $q(\bm{x})$ over $S$. Since $S \cap L$ contains a point of the relative boundary of $S$, this is a contradiction.
\end{proof}

We now prove a few basic convex geometric results.

\begin{lemma} \label{lem:affine-lines}
    Fix a set $S \subseteq (\R^n)^d$ with $d \geq 2$ and any $\bm{w}_3,\ldots,\bm{w}_d \in \R^n$. If $T = \{(\bm{u},\bm{v}) \in (\R^n)^2 : (\bm{u},\bm{v},\bm{w}_3,\ldots,\bm{w}_d) \in S\}$ contains an affine line, then $S$ contains an affine line.
\end{lemma}
\begin{proof}
    By assumption, $(\bm{u}_0,\bm{v}_0) + t \cdot (\bm{u}_1,\bm{v}_1) \in T$ for all $t \in \R$. Therefore $(\bm{u}_0,\bm{v}_0,\bm{w}_3,\ldots,\bm{w}_d) + t \cdot (\bm{u}_1,\bm{v}_1,\bm{0},\ldots,\bm{0}) \in S$ for all $t \in \R$ as desired.
\end{proof}
% \begin{proof}
%     $(\implies)$. This direction is immediate by considering $L^d \subseteq S^d$ for any affine line $L \subseteq S$.

%     $(\impliedby)$. Let $L \subseteq S^d$ be some affine line. Thus there exists $(\bm{a}_1,\ldots,\bm{a}_d) \neq (\bm{b}_1,\ldots,\bm{b}_d)$ in $S^d$ such that
%     \[
%         t \cdot (\bm{a}_1,\ldots,\bm{a}_d) + (1-t) \cdot (\bm{b}_1,\ldots,\bm{b}_d) \in S^d
%     \]
%     for all $t \in \R$. Let $i \in [d]$ be such that $\bm{a}_i \neq \bm{b}_i$, and thus
%     \[
%         t \cdot \bm{a}_i + (1-t) \cdot \bm{b}_i \in S
%     \]
%     for all $t \in \R$. Therefore $S$ contains an affine line.
% \end{proof}

\begin{lemma} \label{lem:rel-boundary}
    Fix a closed convex set $S \subseteq (\R^n)^d$ with $d \geq 2$ and any $\bm{w}_3,\ldots,\bm{w}_d \in \R^n$. If $(\bm{u},\bm{v})$ is on the relative boundary of $T = \{(\bm{x},\bm{y}) \in (\R^n)^2 : (\bm{x},\bm{y},\bm{w}_3,\ldots,\bm{w}_d) \in S\}$, then $(\bm{u},\bm{v},\bm{w}_3,\ldots,\bm{w}_d)$ is on the relative boundary of $S$.
\end{lemma}
\begin{proof}
    Fix $(\bm{u},\bm{v})$ on the relative boundary of $T$. Let $B \subset \mathrm{aff}(S)$ be some small relatively open ball centered at $(\bm{u},\bm{v},\bm{w}_3,\ldots,\bm{w}_d)$, and define
    \[
        B' := \{(\bm{x},\bm{y}) : (\bm{x},\bm{y},\bm{w}_3,\ldots,\bm{w}_d) \in B\}.
    \]
    Thus $B'$ is relatively open in $\mathrm{aff}(T)$ and contains $(\bm{u},\bm{v})$. Since $(\bm{u},\bm{v})$ is on the relative boundary of $T$, there exists $(\bm{u}',\bm{v}') \in B' \setminus T$. Therefore $(\bm{u}',\bm{v}',\bm{w}_3,\ldots,\bm{w}_d) \in B \setminus S$, which implies $B \not\subset S$. Since this holds for any small ball about $(\bm{u},\bm{v},\bm{w}_3,\ldots,\bm{w}_d)$, we must have that $(\bm{u},\bm{v},\bm{w}_3,\ldots,\bm{w}_d)$ is on the relative boundary of $S$.
\end{proof}

Finally, we now prove the main result of the section. First we need a definition.

\begin{definition}
    We call a set $S \subseteq (\R^n)^d$ \textbf{$S_d$-invariant} if for all $(\bm{w}_1,\ldots,\bm{w}_d) \in S$ and all $\sigma \in S_d$ we have that $(\bm{w}_{\sigma(1)},\ldots,\bm{w}_{\sigma(d)}) \in S$.
\end{definition}

\begin{theorem} \label{thm:diagonal-minimizers}
    Let $P(\bm{x}_1,\ldots,\bm{x}_d)$ be a symmetric multilinear form on $\R^n$ for some $d \geq 2$, and let $S \subset (\R^n)^d$ be an $S_d$-invariant closed convex set containing no affine lines. If $\min\{P(\bm{x}_1,\ldots,\bm{x}_d) : (\bm{x}_1,\ldots,\bm{x}_d) \in S\}$ has no minimizers on the relative boundary of $S$, then all its minimizers over $S$ are of the form $(\bm{w},\ldots,\bm{w})$ for some $\bm{w} \in \R^n$.
\end{theorem}
\begin{proof}
    If $P$ has no minimizer in $S$, then we are done. Otherwise suppose $(\bm{w}_1,\ldots,\bm{w}_d)$ is a minimizer of $P$ over $S$, which by assumption is contained in the relative interior of $S$. Define a function
    \[
        f(\bm{u},\bm{v}) := P(\bm{u},\bm{v},\bm{w}_3,\ldots,\bm{w}_d),
    \]
    which is a symmetric bilinear form on $\R^n$. Further, one can consider $q(\bm{u},\bm{v}) = f(\bm{u},\bm{v})$ to be a quadratic form on $\R^{2n}$. Thus $(\bm{w}_1,\bm{w}_2)$ is a minimizer of $q$ over $T = \{(\bm{u},\bm{v}) \in (\R^n)^2 : (\bm{u},\bm{v},\bm{w}_3,\ldots,\bm{w}_d) \in S\}$.

    Suppose now that $\min\{q(\bm{x},\bm{y}) : (\bm{x},\bm{y}) \in T\}$ has a minimizer $(\bm{u}',\bm{v}')$ on the relative boundary of $T$. Then $(\bm{u}',\bm{v}',\bm{w}_3,\ldots,\bm{w}_d)$ is a minimizer of $P$ over $S$, and also $(\bm{u}',\bm{v}',\bm{w}_3,\ldots,\bm{w}_d)$ is contained in the relative boundary of $S$ by \cref{lem:rel-boundary}. This contradicts the assumptions, and thus $\min\{q(\bm{x},\bm{y}) : (\bm{x},\bm{y}) \in T\}$ has no minimizers on the relative boundary of $T$. Thus by \cref{lem:at-most-one-min} and \cref{lem:affine-lines}, $\min\{q(\bm{x},\bm{y}) : (\bm{x},\bm{y}) \in T\}$ has at most one minimizer in $T$, and therefore $(\bm{w}_1,\bm{w}_2)$ is the unique minimizer of $q$ over $T$.

    But by symmetry of $P$ and $S_d$-invariance of $S$, $(\bm{w}_2,\bm{w}_1)$ is also a minimizer of $q$ over $T$, and thus $\bm{w}_1 = \bm{w}_2$. By applying the same argument to $(\bm{w}_i,\bm{w}_j)$ for all $i \neq j$, we find that there is some $\bm{w} \in S$ such that $\bm{w}_i = \bm{w}$ for all $i$. Thus all minimizers of $P$ are of the form $(\bm{w},\ldots,\bm{w})$ for some $\bm{w} \in \R^n$.
\end{proof}

\section{Unique Minimizer: Mixed Discriminant} \label{sec:proof-MD-rank-one-minimizer}

% - mixed form is inner product of p and product of linears
% - mixed discriminant with better bound: sum of x_i A_i, diagonalize this matrix, diagonal exrpessable as x_i * diag(UA_iU*), row sums equal alpha, majorized by eigenvalues, now plug in alpha into bound, use triangle ineq, get a different bound for each x
% - just write it more mixed discriminant, also Hermitian
% - rectangular permanent follows from adding all ones to matrix
% - only geometric stuff to state for mixed discriminant: linear space of tuple of vectors x_1,...,x_n with property that sum is 0 and <u,y_i> = 0 for some u, orthogonal complement of this is (c + delta_i)_i where each delta_i is proportional to u

% - eigenvalues of diagonal D are majorized by diagonals of A_j? row sums are eigenvalues of D
% - matrix capacity is same as transpose

In this section we prove our first main result \Cref{main-MD-unique-min} on the unique minimization of the mixed discriminant. Recall that the \textbf{mixed discriminant} is defined via
\[
    \MD(M_1,\ldots,M_n) = \frac{1}{n!} \partial_{x_1} \cdots \partial_{x_n} \det\left(x_1 A_1 + \cdots + x_n A_n\right)
\]
for a given tuple of $n \times n$ matrices $(A_1,\ldots,A_n)$. Our main result can then be re-stated as follows.

\begin{theorem} \label{section-md-unique-min}
    For all $n$ there exists $\epsilon > 0$ such that: for all $C \in \PSD_n$ with eigenvalues $\bm{c}$ such that $\tr(C) = n$ and $\|\bm{c}-\bm{1}\|_1 < \epsilon$, the mixed discriminant is uniquely minimized over $\MatTup_n(C)$ at $\frac{1}{n} (C,C,\ldots,C)$.

    More concretely, the result holds for any $C \in \PSD_n$ with eigenvalues $\bm{c}$ such that $\tr(C) = n$ and
    \[
        \frac{1}{c_1 \cdots c_n} \left(1 - \frac{\|\bm{c}-\bm{1}\|_1}{2}\right)^n > \frac{(n-1)^{2n-3}}{n^{n-1}(n-2)^{n-2}}.
    \]
\end{theorem}

We now prove \Cref{section-md-unique-min}. Fix $(A_1,\ldots,A_n) \in \MatTup_n(C)$. We consider
\[
    p(x_1,\ldots,x_n) = \det\left(\sum_{i=1}^n x_i A_i\right),
\]
where the mixed discriminant of $A_1,\ldots,A_n$ is proportional to the all-ones coefficient of $p$.

\paragraph{Step 1: Lower bound the capacity of $p$.} Fix $x_1,\ldots,x_n \in \R$. Since $\sum_{i=1}^n x_i A_i$ is Hermitian, there exist unitary $U$ and diagonal $D$ (dependent on $x$) such that
\[
    UDU^* = \sum_{i=1}^n x_i A_i.
\]
This implies
\[
    D = \sum_{i=1}^n x_i U^*A_iU,
\]
which gives
\[
    p(x_1,\ldots,x_n) = \prod_{i=1}^n \sum_{j=1}^n M_{ij} x_j
\]
where $M_{ij} = (U^*A_jU)_{ii} \geq 0$ since $A_j$ is PSD for all $j \in [n]$. Note that for each $j$ we have
\[
    \sum_{i=1}^n M_{ij} = \sum_{i=1}^n (U^*A_jU)_{ii} = \tr(U^*A_jU) = 1,
\]
and for each $i$ we have
\[
    \sum_{j=1}^n M_{ij} = \sum_{j=1}^n (U^*A_jU)_{ii} = (U^*CU)_{ii}.
\]
Thus the column sums of $M$ are $\bm{1}$ and the row sums $\bm{a}$ of $M$ are \textbf{majorized} by $\bm{c}$ (see condition $(1)$ of \Cref{majorization} below). We recall some equivalent definitions of majorization as follows. We will not use all of these, but have decided to include them for the interested reader.

\begin{definition}[e.g. \cite{MarshallOlkinArnold2011}] \label{majorization}
    Given $\bm{x},\bm{y} \in \R^n$, we say that $\bm{y}$ \textbf{majorizes} $\bm{x}$, equivalently written $\bm{x} \preceq \bm{y}$ if one of the following equivalent conditions holds:
    \begin{enumerate}
        \item there is a Hermitian matrix with diagonal entries $\bm{x}$ and eigenvalues $\bm{y}$,
        \item $\|\bm{x}-\alpha\bm{1}\|_1 \leq \|\bm{y}-\alpha\bm{1}\|_1$ for all $\alpha \in \R$,
        \item $\sum_{i=1}^n x_i = \sum_{i=1}^n y_i$ and $\sum_{i=1}^k x_i^{\downarrow} \leq \sum_{i=1}^k y_i^{\downarrow}$ for all $k \in [n]$, where $\bm{x}^{\downarrow}$ is $\bm{x}$ with entries sorted in non-increasing order,
        \item there exists $D \in \Mat_n(\bm{1})$ such that $D\bm{y} = \bm{x}$, and
        \item $\conv\left(\sigma \cdot \bm{x} : \sigma \in S_n\right) \subseteq \conv\left(\sigma \cdot \bm{y} : \sigma \in S_n\right)$.
    \end{enumerate}
\end{definition}

Note that condition $(2)$ of \Cref{majorization} implies $\|\bm{a}-\bm{1}\|_1 \leq \|\bm{c}-\bm{1}\|_1$. Therefore by combining the above with \Cref{cap-bound-transpose} and \Cref{thm:rs-cap-bound}, for all $\bm{x} \in \R_{>0}^n$ we have
\begin{align*}
    \frac{p(\bm{x})}{x_1 \cdots x_n} &= \frac{\prod_{i=1}^n (M\bm{x})_i}{x_1 \cdots x_n} \\
        &\geq \cpc_{\bm{1}}\left(\prod_{i=1}^n (M\bm{x})_i\right) = \cpc_{\bm{1}}\left(\prod_{i=1}^n (M^\top\bm{x})_i\right) \\
        &\geq \left(1 - \frac{\|\bm{a}-\bm{1}\|_1}{2}\right)^n \geq \left(1 - \frac{\|\bm{c}-\bm{1}\|_1}{2}\right)^n
\end{align*}
whenever $\|\bm{c}-\bm{1}\|_1 < 2$. This implies
\begin{equation} \label{eq:MD-cap-bound}
    \cpc_{\bm{1}}(p) \geq \left(1 - \frac{\|\bm{c}-\bm{1}\|_1}{2}\right)^n.
\end{equation}

\paragraph{Step 2: Apply \Cref{thm:diagonal-minimizers}}

We note that $\MD$ is a symmetric multilinear form, and that $\MatTup_n(C)$ is a compact convex $S_n$-invariant set. Thus to apply \Cref{thm:diagonal-minimizers}, we need to study the relative boundary of $\MatTup_n(C)$. A tuple $(A_1,\ldots,A_n)$ is on the relative boundary of $\MatTup_n(C)$ if and only if $A_i$ is singular for some $i \in [n]$. In this case, we have
\[
    p(\bm{e}_i) = \det(A_i) = 0,
\]
which by \Cref{boundary-cap-bound} implies
\begin{equation} \label{eq:MD-rel-boundary}
    \frac{\MD(A_1, \ldots, A_n)}{\cpc_{\bm{1}}(p)} \geq \frac{(n-2)^{n-2}}{n(n-1)^{2n-3}}.
\end{equation}
On the other hand, we compute
\[
    \MD\left(\frac{C}{n}, \ldots, \frac{C}{n}\right) = \frac{\det(C)}{n^n} = \frac{c_1 \cdots c_n}{n^n}.
\]
Now if $\bm{c}$ is such that
\[
    \frac{(n-2)^{n-2}}{n(n-1)^{2n-3}} \left(1 - \frac{\|\bm{c}-\bm{1}\|_1}{2}\right)^n > \frac{c_1 \cdots c_n}{n^n},
\]
we combine \eqref{eq:MD-cap-bound} and \eqref{eq:MD-rel-boundary} to obtain
\[
    \MD(A_1,\ldots,A_n) > \frac{c_1 \cdots c_n}{n^n} = \MD\left(\frac{C}{n}, \ldots, \frac{C}{n}\right)
\]
for all $(A_1,\ldots,A_n)$ on the relative boundary of $\MatTup_n(C)$. Therefore we can apply \Cref{thm:diagonal-minimizers} to conclude that all minimizers of $\MD$ over $\MatTup_n(C)$ are of the form $(A,\ldots,A)$. The only such element of $\MatTup_n(C)$ is $\frac{1}{n} (C,\ldots, C)$, and thus this uniquely minimizes $\MD$ over $\MatTup_n(C)$.

\section{Unique Minimizer: SLC Polynomials} \label{sec:proof-SLC-rank-one-minimizer}

The point of this section is to prove \Cref{main-SLC-unique-min}; that is, to generalize to strongly log-concave polynomials the fact that the permanent is uniquely minimized over doubly stochastic matrices by the multiple of the all-ones matrix. An analogous fact was proven for doubly stochastic real stable polynomials in \cite{Gur07VdW}. Let us explicitly recall \Cref{main-SLC-unique-min}.

\begin{theorem} \label{section-SLC-unique-min}
    Let $p \in \R_{\geq 0}[x_1,\ldots,x_n]$ be $n$-homogeneous and SLC. If
    \[
        \partial_{x_1} \cdots \partial_{x_n} p(\bm{x}) = \frac{n!}{n^n} \cpc_{\bm{1}}(p) > 0,
    \]
    then for $\bm{v} = \frac{\nabla p(\bm{1})}{n}$ we have $p(\bm{x}) = \left(\sum_i v_i x_i\right)^n$.
\end{theorem}

\subsection{Univariate $n$-Newton Polynomials}

To prove \Cref{section-SLC-unique-min}, we first need a number of facts about the class of $n$-Newton univariate polynomials derived from SLC polynomials. Recall the definition of $n$-Newton in \Cref{def:n-Newton}. The main result of this section is \Cref{univar-Newton-cap}, which gives a refined capacity bound for $n$-Newton polynomials based on their degree. To prove this, we first need a few basic results on hte capacity of $n$-Newton polynomials.

% \begin{lemma} \label{slc-univariate-eval}
%     If $p \in \R_{\geq 0}[x_1,\ldots,x_n]$ is $n$-homogeneous and SLC, then $f(t) = p(x_1,\ldots,x_{n-1},t)$ is $n$-Newton for all $\bm{x} \in \R_{\geq 0}^{n-1}$.
% \end{lemma}
% \begin{proof}
%     TODO: see multivar Newton ineq, see other recent stuff too
% \end{proof}

\begin{lemma} \label{univar-power-of-linears-cap}
    For any $a > 0$, we have that
    \[
        \cpc_1\left(\left(1+\frac{a}{n} t\right)^n\right) = a\left(\frac{n}{n-1}\right)^{n-1},
    \]
    with the optimum achieved at $t = \frac{n}{a(n-1)}$.
\end{lemma}
\begin{proof}
    We compute
    \[
        \partial_t \left[\frac{1}{t} \left(1 + \frac{a}{n} t\right)^n\right] = \frac{1}{t^2}\left(1 + \frac{a}{n}\right)^{n-1} \left(at \cdot \frac{n-1}{n} -  1\right),
    \]
    and thus the derivative is $0$ exactly at $t = \frac{n}{a(n-1)}$. Therefore $\cpc_1\left(\left(1+\frac{a}{n}t\right)^n\right)$ is optimized at $t = \frac{n}{a(n-1)}$ and
    \[
        \cpc_1\left(\left(1 + \frac{a}{n} t\right)^n\right) = \frac{a(n-1)}{n} \left(\frac{n}{n-1}\right)^n = a \left(\frac{n}{n-1}\right)^{n-1}.
    \]
\end{proof}

\begin{lemma} \label{strict-cap-comparison}
    Let $f,g \in \R_{\geq 0}[t]$ be such that $f$ is of degree at least $2$ and $g(0) > 0$. Then $\cpc_1(f)$ is optimized at some $t_0 > 0$, and if $g(t) < f(t)$ for all $t > 0$ then $\cpc_1(g) < \cpc_1(f)$.
\end{lemma}
\begin{proof}
    The fact that $\lim_{t \to 0^+} \frac{f(t)}{t} = +\infty = \lim_{t \to +\infty} \frac{f(t)}{t}$ implies the first claim. Now let $t_0 > 0$ be a value of $t$ which optimizes $\cpc_1(f)$. Then we have
    \[
        \cpc_1(f) = \frac{f(t_0)}{t_0} > \frac{g(t_0)}{t_0} \geq \cpc_1(g),
    \]
    which implies the second claim.
\end{proof}

We now prove the mian result of the section. To do this, we make use of the \textbf{truncation} operator, which ``cuts off'' terms of a polynomial above a certain degree. Formally, for $d \leq n$ we define
\[
    \mathrm{trunc}_d\left(\sum_{k=0}^n c_k t^k\right) = \sum_{k=0}^d c_k t^k.
\]

\begin{theorem} \label{univar-Newton-cap}
    For all $d,n$ such that $1 \leq d \leq n$ there exist constants $\gamma_{n,d} \geq 1$ such that:
    \begin{enumerate}
        \item for all $n$-Newton polynomials $f(t)$ of degree at most $d$ such that $\partial_t f(0) > 0$ we have
        \[
            \partial_t f(0) \geq \gamma_{n,d} \left(\frac{n-1}{n}\right)^{n-1} \cpc_1(f),
        \]
        \item $1 = \gamma_{n,n} < \gamma_{n,n-1} < \cdots < \gamma_{n,2} < \gamma_{n,1} = \left(\frac{n}{n-1}\right)^{n-1}$, and
        \item assuming $d \geq 2$ (and $\partial_t f(0) > 0$), the inequality of $(1)$ is an equality if and only if $f(t) = \mathrm{trunc}_d\left[(a+bt)^n\right]$ for some $a,b > 0$.
    \end{enumerate}
\end{theorem}
\begin{proof}
    We define
    \[
        \gamma_{n,d} = \frac{\cpc_1\left(\left(1 + t\right)^n\right)}{\cpc_1\left(\mathrm{trunc}_d\left[\left(1 + t\right)^n\right]\right)}.
    \]
    Note that $\gamma_{n,n} = 1$ and $\gamma_{n,1} = \left(\frac{n}{n-1}\right)^{n-1}$ by \Cref{univar-power-of-linears-cap}, and $\gamma_{n,n} < \gamma_{n,n-1} < \cdots < \gamma_{n,2} < \gamma_{n,1}$ by \Cref{strict-cap-comparison}. This implies $(2)$ and the $(\impliedby)$ direction of $(3)$.

    Now if $d=1$ or $f(0)=0$ then $\partial_t f(0) = \cpc_1(f)$. Thus for the remainder of the proof, we may assume without loss of generality that $f(0) = 1$ and $d \geq 2$. Since $f$ is $n$-Newton we have
    \begin{align*}
        f(t) &= \binom{n}{0} + \binom{n}{1} \frac{\partial_t f(0)}{n} t + \sum_{k=2}^d \binom{n}{k} c_k t^k \\
            &\leq \binom{n}{0} + \binom{n}{1} \frac{\partial_t f(0)}{n} t + \sum_{k=2}^d \binom{n}{k} \left(\frac{\partial_t f(0)}{n}\right)^k t^k \\
            &= \mathrm{trunc}_d\left[\left(1 + \frac{\partial_t f(0)}{n} t\right)^n\right].
    \end{align*}
    Note that if $f(t)$ is not equal to $g(t) := \mathrm{trunc}_d\left[\left(1 + \frac{\partial_t f(0)}{n} t\right)^n\right]$ then the above inequality is strict for all $t > 0$. Thus by \Cref{strict-cap-comparison} we have
    \[
        \cpc_1(f) \leq \cpc_1\left(\mathrm{trunc}_d\left[\left(1 + \frac{\partial_t f(0)}{n} t\right)^n\right]\right),
    \]
    with strict inequality unless $f(t) = \mathrm{trunc}_d\left[\left(1 + \frac{\partial_t f(0)}{n} t\right)^n\right]$. By \Cref{univar-power-of-linears-cap}, we also have $\cpc_1\left(\left(1 + \frac{\partial_t f(0)}{n} t\right)^n\right) = \partial_t f(0) \left(\frac{n}{n-1}\right)^{n-1}$. Since
    \begin{align*}
        \frac{\cpc_1\left(\left(1 + \frac{\partial_t f(0)}{n} t\right)^n\right)}{\cpc_1\left(\mathrm{trunc}_d\left[\left(1 + \frac{\partial_t f(0)}{n} t\right)^n\right]\right)} &= \frac{\frac{\partial_t f(0)}{n} \cpc_1\left(\left(1 + t\right)^n\right)}{\frac{\partial_t f(0)}{n} \cpc_1\left(\mathrm{trunc}_d\left[\left(1 + t\right)^n\right]\right)} \\
            &= \frac{\cpc_1\left(\left(1 + t\right)^n\right)}{\cpc_1\left(\mathrm{trunc}_d\left[\left(1 + t\right)^n\right]\right)} \\
            &= \gamma_{n,d},
    \end{align*}
    we therefore have
    \begin{align*}
        \partial_t f(0) \left(\frac{n}{n-1}\right)^{n-1} &= \cpc_1\left(\left(1 + \frac{\partial_t f(0)}{n} t\right)^n\right) \\
            &= \gamma_{n,d} \cpc_1\left(\mathrm{trunc}_d\left[\left(1 + \frac{\partial_t f(0)}{n} t\right)^n\right]\right) \\
            &\geq \gamma_{n,d} \cpc_1(f),
    \end{align*}
    with strict inequality unless $f(t) = \mathrm{trunc}_d\left[\left(1 + \frac{\partial_t f(0)}{n} t\right)^n\right]$. By rearranging, this implies $(1)$ and also the $(\implies)$ direction of $(3)$.
\end{proof}

\subsection{Support and Capacity of SLC Polynomials}

In this section we prove various results on the interaction between support, capacity bounds, and capacity minimizers for SLC polynomials. The main result of this section is \Cref{full-support-from-derivs}, which gives a capacity-based sufficient condition for an SLC polynomial to have full support. Note that we do not directly use \Cref{full-support-from-derivs} in \Cref{subsec:proof-SLC-unique-min}, but instead give the argument used in the proof. We still give \Cref{full-support-from-derivs} explicitly to unify this section.

First is a more refined version of a derivative capacity bound due to Gurvits, which was originally used to prove \Cref{thm:Gurvits-SLC}.

\begin{theorem} \label{Gurvits-strengthening}
    Let $p \in \R_{\geq 0}[x_1,\ldots,x_n]$ be $n$-homogeneous and SLC, and let $d_i$ be the degree of $x_i$ in $p$. If $\partial_{x_1} \cdots \partial_{x_n} p(\bm{x}) > 0$ then
    \[
        \cpc_{\bm{1}}\left(\left.\partial_{x_i}\right|_{x_i=0} p(\bm{x})\right) \geq \gamma_{n,d_i} \left(\frac{n-1}{n}\right)^{n-1} \cpc_{\bm{1}}(p),
    \]
    where $\gamma_{n,d_i} \geq 1$ is as defined in \Cref{univar-Newton-cap}.
\end{theorem}
\begin{proof}
    Without loss of generality, we assume $i=n$. Thus for all $\bm{x} \in \R_{>0}^{n-1}$, the polynomial $f(t) = p(x_1,\ldots,x_{n-1},t)$ is of degree $d_n$. Further, $\partial_t f(0) > 0$ by assumption, and $f$ is $n$-Newton by \Cref{slc-univariate-eval}.

    With this, by \Cref{univar-Newton-cap} we have
    \[
        \partial_t f(0) \geq \gamma_{n,d_n} \left(\frac{n-1}{n}\right)^{n-1} \cpc_1(f)
    \]
    which implies
    \[
        \cpc_{\bm{1}}\left(\left.\partial_{x_n}\right|_{x_n=0} p(\bm{x})\right) \geq \gamma_{n,d_n} \left(\frac{n-1}{n}\right)^{n-1} \cpc_{\bm{1}}(p)
    \]
    by taking $\inf$ over $\bm{x} \in \R_{>0}^{n-1}$ on both sides.
\end{proof}

Now we prove a basic support lemma for SLC polynomials. This can be proven using the fact that SLC polynomials have M-convex support, but we have decided to give a simpler self-contained proof here.

\begin{lemma} \label{slc-support}
    Let $p \in \R_{\geq 0}[x_1,\ldots,x_n]$ be $n$-homogeneous and SLC. If the coefficient of $x_i^n$ in $p$ is positive for all $i \in [n]$ then $p$ has full support.
\end{lemma}
\begin{proof}
    We prove this result by induction on $n$. Fix $p$ and define
    \[
        H_p = \{\bm\alpha \in \Delta_n \cap \Z^n : \langle \bm{x}^{\bm\alpha} \rangle p(\bm{x}) = 0\}.
    \]
    So as to obtain a contradiction, suppose $H_p$ is non-empty and let $\bm\alpha$ be the vector which maximizes $\max_{i \in [n]} \alpha_i$ over $H_p$. By possibly permuting variables, we will assume that $\alpha_1 \geq \alpha_2 \geq \cdots \geq \alpha_n$.

    \paragraph{Case 1: $\alpha_3 = 0$ or $n \leq 2$.} In this case $f(t) = p(t,1,0,\ldots,0)$ is $n$-Newton and has strictly positive $t^0$ and $t^n$ coefficients by assumption. The $n$-Newton property then implies the $t^k$ coefficient of $f$ is positive for all $0 \leq k \leq n$, which is a contradiction.

    \paragraph{Case 2: $\alpha_3 > 0$.} In this case $\alpha_1 \leq n-2$, which implies the $x_i^{n-1} x_j$ coefficient of $p$ is strictly positive for all $i,j$ by our assumption on $\bm\alpha$. Thus $\left.\partial_{x_1}\right|_{x_1=0} p(\bm{x})$ is $(n-1)$-variate and $(n-1)$-homogeneous, with strictly positive $x_i^{n-1}$ coefficients for all $i \in [n]$. The inductive hypothesis then implies the $\bm{x}^{\bm\alpha}$ coefficient of $p$ is strictly positive, which is a contradiction.
\end{proof}

The above two results combine to give the mian result of the section.

\begin{corollary} \label{full-support-from-derivs}
    Let $p \in \R_{\geq 0}[x_1,\ldots,x_n]$ be $n$-homogeneous and SLC. If $\partial_{x_1} \cdots \partial_{x_n} p(\bm{x}) > 0$ and
    \[
        \cpc_{\bm{1}}\left(\left.\partial_{x_i}\right|_{x_i=0} p(\bm{x})\right) = \left(\frac{n-1}{n}\right)^{n-1} \cpc_{\bm{1}}(p)
    \]
    for all $i \in [n]$, then $p$ has full support.
\end{corollary}
\begin{proof}
    So as to obtain a contradiction, suppose $p$ does not have full support. Then by \Cref{slc-support}, there exists $i \in [n]$ such that the coefficient of $x_i^n$ is $0$ in $p$. Thus by \Cref{Gurvits-strengthening} we have
    \[
        \cpc_{\bm{1}}\left(\left.\partial_{x_i}\right|_{x_i=0} p(\bm{x})\right) \geq \gamma_{n,d_i} \left(\frac{n-1}{n}\right)^{n-1} \cpc_{\bm{1}}(p),
    \]
    where $d_i < n$. Since $\gamma_{n,d_i} > 1$ by \Cref{univar-Newton-cap}, this is a contradiction.
\end{proof}

% reduce to DS:
% unique minimzer capacity from full support
% elementary proof (paper with Samorodnitsky) - set of inputs prod = 1 with some upper bound on capacity is compact (t1/t2 vs t2/t1)
% scale to DS

\subsection{General Capacity Minimizers}

This section is devoted to some basic results on capacity minimizers and uniqueness, most of which can be found in \cite{GurvitsSamorodnitsky2000}. Note that none of these results require the SLC property. Recall that a homogeneous polynomial $p \in \R_{\geq 0}[x_1,\ldots,x_n]$ is \textbf{doubly stochastic} if $p(\bm{1}) = 1$ and $\nabla p(\bm{1}) = \bm{1}$.

\begin{lemma} \label{compact-capacity}
    Let $p \in \R_{\geq 0}[x_1,\ldots,x_n]$ be $n$-homogeneous with full support. Then for all $c > \cpc_{\bm{1}}(p)$ we have that $K_c := \{\bm{x} > 0 : \prod_i x_i = 1, ~ p(\bm{x}) \leq c\}$ is compact. In particular, $\cpc_{\bm{1}}(p)$ is minimized at some $\bm{y} > \bm{0}$.
\end{lemma}
\begin{proof}
    Let $p(\bm{x}) = \sum_{\bm\kappa} p_{\bm\kappa} \bm{x}^{\bm\kappa}$, and let $p_{\min} > 0$ be the minimum coefficient of $p$. Fix $c > \cpc_{\bm{1}}(p)$. For any $i \neq j$ and any $\bm{x} \in K_c$ we have
    \[
        c \geq p(\bm{x}) = \sum_{\bm\kappa} p_{\bm\kappa} \bm{x}^{\bm\kappa} \geq \bm{x}^{\bm{1}} \cdot \frac{x_i}{x_j} \cdot p_{\bm{1}+\bm{e}_i-\bm{e}_j} \geq \frac{x_i}{x_j} \cdot p_{\min}.
    \]
    Thus $\max_{i,j}\left(\frac{x_i}{x_j}\right) \leq \frac{c}{p_{\min}}$, which implies $x_i \in [\frac{p_{\min}}{c}, \frac{c}{p_{\min}}]$ since $\prod_i x_i = 1$. Therefore
    \[
        K_c = \left\{\bm{x} \in \left[\frac{p_{\min}}{c}, \frac{c}{p_{\min}}\right]^n : \prod_i x_i = 1, ~ p(\bm{x}) \leq c\right\},
    \]
    which implies $K_c$ is compact.
\end{proof}

\begin{lemma} \label{ds-cap-min-at-1}
    Let $p \in \R_{\geq 0}[x_1,\ldots,x_n]$ be such that $p(\bm{1}) = 1$. Then $p$ is doubly stochastic if and only if $\bm{1}$ minimizes $\cpc_{\bm{1}}(p)$.
\end{lemma}
\begin{proof}
    First, $\bm{1}$ minimizes $\cpc_{\bm{1}}(p)$ if and only if
    \[
        \inf_{\bm{t} \in \R^n} \left[\log p(e^{\bm{t}}) - \langle \bm{t}, \bm{1} \rangle\right]
    \]
    is minimized at $\bm{t} = \bm{0}$. Since $\log p(e^{\bm{t}}) - \langle \bm{t}, \bm{1} \rangle$ is convex and smooth, this is equivalent to $\left.\nabla\right|_{\bm{t}=\bm{0}} \log p(e^{\bm{t}}) = \bm{1}$, where $\left.\nabla\right|_{\bm{t}=\bm{0}} \log p(e^{\bm{t}}) = \nabla p(\bm{1})$.
    This is the same as saying $p$ is doubly stochastic.
\end{proof}

\begin{lemma} \label{ds-scaling}
    Let $p \in \R_{\geq 0}[x_1,\ldots,x_n]$ be $n$-homogeneous with full support. Then there exists $\bm{y} > \bm{0}$ such that $\frac{p(\bm{y} \odot \bm{x})}{p(\bm{y})}$ is doubly stochastic.
\end{lemma}
\begin{proof}
    By \Cref{compact-capacity}, there exists some $\bm{y} > \bm{0}$ which minimizes $\cpc_{\bm{1}}(p)$. Letting $q(\bm{x}) = \frac{p(\bm{y} \odot \bm{x})}{p(\bm{y})}$, we have $q(\bm{1}) = 1$. Since $\bm{y}$ minimizes $\cpc_{\bm{1}}(p)$, we have that $\bm{1}$ minimizes $\cpc_{\bm{1}}(q)$. Therefore $q$ is doubly stochastic by \Cref{ds-cap-min-at-1}.
\end{proof}

% uniqueness of minimizer for DS with p(1) = 1:
% log sum x_i p_i > sum p_i log x_i strictly whenever x_i not all ones (for log and positive p_i probabilities)
% log p(x) = log sum p_alpha x^alpha > sum p_alpha log x^alpha strict unless x^alpha all equal
% Euler identity and DS finishes the proof

\begin{lemma} \label{ds-cap-unique-min}
    Let $p \in \R_{\geq 0}[x_1,\ldots,x_n]$ be $n$-homogeneous with full support. If $p$ is doubly stochastic, then $p$ is uniquely minimized over $\{\bm{x} > 0 : \prod_i x_i = 1\}$ at $\bm{x} = \bm{1}$.
\end{lemma}
\begin{proof}
    Let $p(\bm{x}) = \sum_{\bm\kappa} p_{\bm\kappa} \bm{x}^{\bm\kappa}$. Using Jensen's inequality, $p(\bm{1}) = 1$ implies
    \[
        \log p(\bm{x}) = \log \left(\sum_{\bm\kappa} p_{\bm\kappa} \bm{x}^{\bm\kappa}\right) \geq \sum_{\bm\kappa} p_{\bm\kappa} \log\left(\bm{x}^{\bm\kappa}\right).
    \]
    Since $\log$ is strictly concave, this inequality is strict unless $\bm{x}^{\bm\kappa} = \bm{x}^{\bm\kappa'}$ for all $\bm\kappa,\bm\kappa'$. Thus the above inequality is strict unless $\bm{x} = \bm{1}$. We now further have
    \[
        \sum_{\bm\kappa} p_{\bm\kappa} \log\left(\bm{x}^{\bm\kappa}\right) = \sum_{i=1}^n \log x_i \sum_{\bm\kappa} \kappa_i p_{\bm\kappa} = \sum_{i=1}^n \log x_i \cdot \partial_{x_i} p(\bm{1}) = 0,
    \]
    since $\partial_{x_i} p(\bm{1}) = 1$ by double stochasticity and $\sum_{i=1}^n \log x_i = 0$. Therefore, $p(\bm{x}) \geq 1$ with strict inequality unless $\bm{x} = \bm{1}$.
\end{proof}

\begin{corollary} \label{general-cap-unique-min}
    Let $p \in \R_{\geq 0}[x_1,\ldots,x_n]$ be $n$-homogeneous with full support. Then $p$ is uniquely minimized over $\{\bm{x} > 0 : \prod_i x_i = 1\}$.
\end{corollary}
\begin{proof}
    By \Cref{ds-scaling}, there exists $\bm{y} > \bm{0}$ such that $q(\bm{x}) = \frac{p(\bm{y} \odot \bm{x})}{p(\bm{y})}$ is doubly stochastic. By \Cref{ds-cap-unique-min}, $q$ is uniquely minimized over $\{\bm{x} > 0 : \prod_i x_i = 1\}$ at $\bm{x} = \bm{1}$. Therefore $p$ is uniquely minimized over $\{\bm{x} > 0 : \prod_i x_i = 1\}$ at $\bm{x} = \bm{y}$.
\end{proof}

As a final remark, we note that \Cref{general-cap-unique-min} can be strengthened to say that we have unique minimization if and only if $p$ has full support. In fact, by the proof of \Cref{compact-capacity}, the support condition can be weakened to just the coefficients ``nearby'' to the all-ones vector. These facts were already proven in the hyperbolic case in \cite{Gur07VdW}.

\subsection{Proof of \Cref{section-SLC-unique-min}} \label{subsec:proof-SLC-unique-min}

% uniqueness of univariate vdw for SLC is different from real stable

% Finally, we can assume (summary):
% - polynomial is DS
% - all coeff >0 strictly
% - p(1,..,1,t) is (a+bt)^n -> values of a, b determined
% -  obtain value at e_i, use argument from pdf in email

% mixed volume uses Brunn-Minkowski uniqueness

Let $p(\bm{x})$ be an $n$-variate $n$-homogeneous SLC polynomial for which $\partial_{x_1} \cdots \partial_{x_n} p = \frac{n!}{n^n} \cpc_{\bm{1}}(p) > 0$. Define $p_i(\bm{x}) = \left.\partial_{x_i}\right|_{x_i=0} p(\bm{x})$ for all $i \in [n]$, let $d_i$ denote the degree of $x_i$ in $p(\bm{x})$, and let $\bm{v} = \frac{\nabla p(\bm{1})}{n}$. Our goal is to prove that $p(\bm{x}) = \left(\sum_i v_i x_i\right)^n$.

\paragraph{Step 1: Prove $p$ has full support.} By \Cref{Gurvits-strengthening} we have
\[
    \cpc_{\bm{1}}(p_i) \geq \gamma_{n,d_i} \left(\frac{n-1}{n}\right)^{n-1} \cpc_{\bm{1}}(p)
\]
for all $i \in [n]$, where $\gamma_{n,d_i} \geq 1$ is as defined in \Cref{univar-Newton-cap}. Thus for $i = n$, by iteratively using \Cref{Gurvits-strengthening} we obtain
\begin{align*}
    \frac{n!}{n^n} \cpc_{\bm{1}}(p) &= \partial_{x_1} \cdots \partial_{x_n} p \\
        &= \cpc_{\bm{1}}\left(\left.\partial_{x_2}\right|_{x_2=0} \cdots \left.\partial_{x_n}\right|_{x_n=0} p\right) \\
        &\geq \left(\frac{1}{2}\right)^1 \cpc_{\bm{1}}\left(\left.\partial_{x_3}\right|_{x_3=0} \cdots \left.\partial_{x_n}\right|_{x_n=0} p\right) \\
        &\geq \left(\frac{1}{2}\right)^1 \left(\frac{2}{3}\right)^2 \cpc_{\bm{1}}\left(\left.\partial_{x_4}\right|_{x_4=0} \cdots \left.\partial_{x_n}\right|_{x_n=0} p\right) \\
        &\geq \cdots \\
        &\geq \left[\prod_{k=1}^{n-1} \left(\frac{k-1}{k}\right)^{k-1}\right] \cpc_{\bm{1}}\left(\left.\partial_{x_n}\right|_{x_n=0} p\right) \\
        &\geq \gamma_{n,d_n} \left[\prod_{k=1}^n \left(\frac{k-1}{k}\right)^{k-1}\right] \cpc_{\bm{1}}(p) \\
        &= \gamma_{n,d_n} \cdot \frac{n!}{n^n} \cpc_{\bm{1}}(p) \\
        &\geq \frac{n!}{n^n} \cpc_{\bm{1}}(p),
\end{align*}
which implies all inequalities are actually equalities. The same argument works for any ordering of the variables, and thus $\gamma_{n,d_i} = 1$ and
\[
    \cpc_{\bm{1}}(p_i) = \left(\frac{n-1}{n}\right)^{n-1} \cpc_{\bm{1}}(p)
\]
for all $i \in [n]$. Since $\gamma_{n,d_i} = 1$ implies $d_i = n$ by \Cref{univar-Newton-cap}, we therefore have that $p$ has full support by \Cref{slc-support}.

\paragraph{Step 2: Reduce to doubly stochastic $p$.} By \Cref{ds-scaling}, there exists $\bm{y} > 0$ such that $q(\bm{x}) := \frac{p(\bm{y} \cdot \bm{x})}{p(\bm{y})}$ is doubly stochastic. Therefore by assumption we have
\[
    \frac{p(\bm{y})}{y_1 \cdots y_n} \cdot \partial_{x_1} \cdots \partial_{x_n} q(\bm{x}) = \partial_{x_1} \cdots \partial_{x_n} p(\bm{x}) = \frac{n!}{n^n} \cpc_{\bm{1}}(p).
\]
By \Cref{thm:Gurvits-SLC} and \Cref{ds-cap-min-at-1}, we also have
\[
    \partial_{x_1} \cdots \partial_{x_n} q(\bm{x}) \geq \frac{n!}{n^n}
\]
since $q$ is doubly stochastic. Combining these gives
\[
    \frac{n!}{n^n} \cpc_{\bm{1}}(p) = \frac{p(\bm{y})}{y_1 \cdots y_n} \cdot \partial_{x_1} \cdots \partial_{x_n} q(\bm{x}) \geq \frac{n!}{n^n} \cdot \frac{p(\bm{y})}{y_1 \cdots y_n} \geq \frac{n!}{n^n} \cpc_{\bm{1}}(p),
\]
and thus these inequalities must actually be equalities. Therefore
\[
    \partial_{x_1} \cdots \partial_{x_n} q(\bm{x}) = \frac{n!}{n^n},
\]
and if the theorem holds for doubly stochastic $q$, we would obtain $q(\bm{x}) = \left(\frac{1}{n} \sum_i x_i\right)^n$. Thus
\[
    p(\bm{x}) = p(\bm{y}) \cdot q(\bm{y}^{-1} \odot \bm{x}) = p(\bm{y}) \cdot \left(\frac{1}{n} \sum_{i=1}^n \frac{x_i}{y_i}\right)^n,
\]
which implies $p$ is a power of a linear form. Therefore we must have $p(\bm{x}) = \left(\sum_i v_i x_i\right)^n$.

\paragraph{Step 3: Prove $p(\bm{e}_i) = \frac{1}{n^n}$ for all $i$.} From now on we assume that $p$ is doubly stochastic. By \Cref{ds-cap-unique-min}, $\cpc_{\bm{1}}(p)$ is uniquely minimized over $\{\bm{x} > 0 : \prod_i x_i = 1\}$ at $\bm{1}$. Similarly, since $p_n$ also has full support and is $(n-1)$-variate and $(n-1)$-homogeneous, we have that $\cpc_{\bm{1}}(p_n)$ is also uniquely minimized over $\{\bm{x} > 0 : \prod_i x_i = 1\}$ at some $\bm{y}$ by \Cref{general-cap-unique-min}.

Now $f(t) = p(y_1,\ldots,y_{n-1},t)$ is $n$-Newton by \Cref{slc-univariate-eval}. Thus by \Cref{univar-Newton-cap} we have
\begin{align*}
    \cpc_{\bm{1}}(p_n) &= \partial_t f(0) \\
        &\geq \left(\frac{n-1}{n}\right)^{n-1} \cpc_1(f) \\
        &\geq \left(\frac{n-1}{n}\right)^{n-1} \cpc_{\bm{1}}(p) \\
        &= \cpc_{\bm{1}}(p_n),
\end{align*}
which implies all inequalities are actually equalities. In particular,
\[
    \cpc_{\bm{1}}(p) = \cpc_1(f) = \cpc_{1}(p(y_1,\ldots,y_{n-1},t)),
\]
and since $\cpc_{\bm{1}}(p)$ is optimized uniquely at $\bm{1}$ over vectors with product equal to $1$, we have $\bm{y} = \bm{1}$ as well. By the above equalities, we also have
\[
    \partial_t f(0) = \left(\frac{n-1}{n}\right)^{n-1} \cpc_1(f)
\]
which implies $f(t) = (a + bt)^n$ for some $a,b > 0$ by \Cref{univar-Newton-cap}. Since $f(1) = p(\bm{1}) = 1$ and $\partial_t f(0) = \left.\partial_{x_n}\right|_{x_n=0} p(\bm{1}) = 1$ by double stochasticity, we have that $a = \frac{n-1}{n}$ and $b = \frac{1}{n}$. Therefore $p(\bm{e}_n) = \langle t^n \rangle f(t) = \frac{1}{n^n}$. Applying a similar argument to $p_i$ implies $p(\bm{e}_i) = \frac{1}{n^n}$ for all $i \in [n]$.

\paragraph{Step 4: Prove $p(\bm{x}) = \left(\frac{x_1 + \cdots + x_n}{n}\right)^n$.} Since $p$ is a homogeneous polynomial, it is sufficient to prove the equality on the standard simplex $\Delta_n$. By concavity of $p^{1/n}$ (see \Cref{alt-slc-def-1/d}), for all $\bm{x} \in \Delta_n$ we have
\[
    p^{1/n}(\bm{x}) \geq \sum_{i=1}^n x_i p^{1/n}(\bm{e}_i) = \frac{1}{n} \sum_{i=1}^n x_i.
\]
Concavity of $p^{1/n}$ also implies
\[
    p^{1/n}(\bm{x}) \leq p^{1/n}(\bm{1}) + \nabla p^{1/n}(\bm{1})^\top (\bm{x}-\bm{1}) = 1 + \frac{1}{n} \bm{1}^\top (\bm{x}-\bm{1}) = \frac{1}{n} \sum_{i=1}^n x_i.
\]
Therefore we have
\[
    \left(\frac{x_1 + \cdots + x_n}{n}\right)^n \leq p(\bm{x}) \leq \left(\frac{x_1 + \cdots + x_n}{n}\right)^n.
\]
Since this holds for all $\bm{x} \in \Delta_n$, we thus have that $p(\bm{x}) = \left(\frac{x_1 + \cdots + x_n}{n}\right)^n$ for all $\bm{x}$.

% Can't work for sum equals I and traces not 1:
% method does not know that sum is identity, and then we can't guarantee rank-one minimizer

\section{Failure of Rank-One Minimizers} \label{sec:proofs-failure-results}

In this section, we prove our main counter results to simple minimizer uniqueness.

\subsection{Permanent, with both marginals non-uniform}

Fix positive $\bm{r},\bm{c} \neq \bm{1}$ such that $\|\bm{r}\|_1 = \|\bm{c}\|_1 = n$, and consider $\Mat_n(\bm{r},\bm{c})$ which is the set of all $n \times n$ matrices with non-negative entries and row sums $\bm{r}$ and column sums $\bm{c}$. Note that this set contains the matrix $\frac{1}{n} \bm{r} \bm{c}^\top$. We will now prove \Cref{prop:diff-r-c}, which says that $\frac{1}{n} \bm{r} \bm{c}^\top$ never minimizes the permanent over $\Mat_n(\bm{r},\bm{c})$ in this case.

Consider the polynomial $\per(X)$ on a matrix of variables. The gradient of this polynomial can be expressed as a matrix, with entries
\[
    \nabla \per(X)_{ij} = \per(X_{ij}),
\]
where $X_{ij}$ is the matrix $X$ with row $i$ and column $j$ removed. In particular, we have
\[
    \nabla \per\left(\frac{1}{n} \bm{r} \bm{c}^\top\right) = \frac{1}{n^n} \prod_{k \neq i} r_i \prod_{k \neq j} c_j = \per\left(\frac{1}{n}\bm{r}\bm{c}^\top\right) \cdot \frac{1}{r_i c_j}.
\]
Now, any minimizer of the permanent which lies in the relative interior of $\Mat_n(\bm{r},\bm{c})$ must have gradient in $W^\perp$ where $W$ is the linear space parallel to the affine span of $\Mat_n(\bm{r},\bm{c})$. The space $W$ consists of all matrices with row and column sums all equal to $0$, and is of dimension $n^2 - (2n-1)$. Thus $W^\perp$ consists of all matrices $M$ with entries of the form
\[
    m_{ij} = a_i + b_j
\]
for some $a_i,b_j \in \R$. (Note that the dimension of this space of matrices is $2n-1$.)

Now we prove \Cref{prop:diff-r-c}. Since $\frac{1}{n} \bm{r} \bm{c}^\top$ lies in the relative interior of $\Mat_n(\bm{r},\bm{c})$, if it is a minimizer then we need $\nabla \per\left(\frac{1}{n} \bm{r} \bm{c}^\top\right) \in W^\perp$ which implies
\[
    \frac{1}{r_ic_j} = a_i + b_j \quad \text{for all $i,j$,}
\]
for some $a_i,b_j \in \R$. This cannot happen if $\bm{r},\bm{c} \neq \bm{1}$, by the following result applied to $\bm{x} = \bm{a}$, $\bm{y} = \bm{1}$, and $\bm{z} = \bm{b}$.

\begin{lemma}
    If for non-zero vectors $\bm{x},\bm{y},\bm{z}$ the matrix $\bm{x}\bm{y}^\top + \bm{y}\bm{z}^\top$ is rank one, then one of $\bm{x},\bm{z}$ is collinear to $\bm{y}$.
\end{lemma}
\begin{proof}
    Suppose $\bm{y},\bm{z}$ are linearly independent. Since $\bm{x}\bm{y}^\top + \bm{y}\bm{z}^\top$ is rank one, there is some $t,s$ not both zero such that $(\bm{x}\bm{y}^\top + \bm{y}\bm{z}^\top)(t \bm{y} + s \bm{z}) = \bm{0}$. This implies
    \[
        \bm{0} = (t\bm{y}^\top\bm{y} + s\bm{y}^\top\bm{z})\bm{x} + (t\bm{z}^\top\bm{y} + s\bm{z}^\top\bm{z})\bm{y}.
    \]
    Thus if $\bm{x},\bm{y}$ are linearly independent, then we have
    \[
        t\bm{y}^\top\bm{y} + s\bm{y}^\top\bm{z} = 0 = t\bm{z}^\top\bm{y} + s\bm{z}^\top\bm{z}.
    \]
    Therefore
    \[
        0 = t\left(t\bm{y}^\top\bm{y} + s\bm{y}^\top\bm{z}\right) + s\left(t\bm{z}^\top\bm{y} + s\bm{z}^\top\bm{z}\right) = (t\bm{y}+s\bm{z})^\top(t\bm{y}+s\bm{z}) > 0,
    \]
    a contradiction. Thus either $\bm{x},\bm{y}$ are linearly dependent, or else $\bm{y},\bm{z}$ are linearly dependent.
\end{proof}

\subsection{Permanent, with marginals far from all-ones}

% \begin{proposition} \label{prop:minimizer-counter}
%     For all $n$ large enough, there exists $\bm{c}$ such that $\|\bm{c} - \bm{1}\|_1 < 2$ and a sparse matrix $M \in \Mat_n(\bm{c})$ (with linearly many non-zero entries) which has smaller permanent than that of $\frac{1}{n} \bm{1} \cdot \bm{c}^\top$.
% \end{proposition}

We now study an example which proves \Cref{prop:boundary-minimizer}. Fix $t > 0$ and $n \in \N$, and define $\epsilon := \frac{1}{n^{1+t}}$. Further define $\bm\alpha \in \R_{>0}^n$ and $\bm{c} \in \R_{>0}^{n+1}$ via
\[
    \bm\alpha := \big(1-\epsilon, 1-\epsilon, \ldots, 1-\epsilon) \in \R_{>0}^n \qquad \text{and} \qquad \bm{c} := \big(1+\alpha_1, \alpha_2, \alpha_3, \ldots, \alpha_n, \sum_{j=1}^n (1-\alpha_j)\big) \in \R_{>0}^{n+1}.
\]
Note that $\|\bm{c}-\bm{1}\|_1 = 1 + (n-2)\epsilon + 1 - n\epsilon = 2(1-\epsilon) < 2$, which by \cref{thm:rs-cap-bound} implies $\per(X) > 0$ for all $X \in \Mat_{n+1}(\bm{c})$. We first have
\[
    \per\left(\frac{1}{n+1} \bm{1} \cdot \bm{c}^\top\right) = \frac{(n+1)!}{(n+1)^{n+1}} (2-\epsilon) (1-\epsilon)^{n-1} n\epsilon = \frac{n!}{n^{(1+t)n}} \cdot \frac{2-n^{-1-t}}{1-n^{-1-t}} \cdot \frac{(n^{1+t}-1)^n}{n^t (n+1)^n}.
\]
Now consider the matrix
\[
    M = \begin{bmatrix}
        1 & 0 & 0 & \cdots & 0 & 0 \\
        1 - \sum_{j=1}^1 (1-\alpha_j) & \sum_{j=1}^1 (1-\alpha_j) & 0 & \cdots & 0 & 0 \\
        0 & 1 - \sum_{j=1}^2 (1-\alpha_j) & \sum_{j=1}^2 (1-\alpha_j) & \cdots & 0 & 0 \\
        0 & 0 & 1 - \sum_{j=1}^3 (1-\alpha_j) & \cdots & 0 & 0 \\
        \vdots & \vdots & \vdots & \ddots & \vdots & \vdots \\
        0 & 0 & 0 & \cdots & \sum_{j=1}^{n-1} (1-\alpha_j) & 0 \\
        0 & 0 & 0 & \cdots & 1 - \sum_{j=1}^n (1-\alpha_j) & \sum_{j=1}^n (1-\alpha_j)
    \end{bmatrix}.
\]
Note that $M \in \Mat_{n+1}(\bm{c})$. The matrix $M$ is upper-triangular, and thus we have
\[
    \per(M) = \prod_{k=1}^n \sum_{j=1}^k (1-\alpha_j) = \prod_{k=1}^n (k\epsilon) = n! \cdot \epsilon^n = \frac{n!}{n^{(1+t)n}}.
\]
We then further have
\[
    \frac{2-n^{-1-t}}{1-n^{-1-t}} \cdot \frac{(n^{1+t}-1)^n}{n^t (n+1)^n} \approx 2n^{(n-1)t} \cdot \frac{(n - n^{-t})^n}{(n+1)^n} \geq 2n^{(n-1)t} \cdot \frac{(n - 1)^n}{(n+1)^n} \approx \frac{2n^{(n-1)t}}{e^2} > 1,
\]
which implies $\per(M) < \per\left(\frac{1}{n+1} \bm{1} \cdot \bm{c}^\top\right)$ for large enough $n$.

% \subsubsection{The $3 \times 3$ case}

% Here we fully analyze the $3 \times 3$ case. Fix $a,b,c > 0$ such that $a+b+c = 3$. Any matrix in $\Mat_2((a,b,c)^\top, \bm{1})$ is of the form
% \[
%     M_x = \begin{bmatrix}
%         x & y & a-x-y \\
%         z & w & b-z-w \\
%         1-x-z & 1-y-w & c-2+x+y+z+w
%     \end{bmatrix}
% \]
% for some $x,y,z,w$.

\section{An Aside on the Quantum Permanent} \label{sec:aside-QP}

In this section, we recall the definition of the quantum permanent, and then derive the weaker bound for \Cref{QP} mentioned above. We then use this bound to also derive some interesting bounds on the mixed discriminant of PSD matrices.

\subsection{The Quantum Permanent}

Recall that a linear operator $T: \R^{n \times n} \to \R^{n \times n}$ is \textbf{completely positive (CP)} if
\[
    T(X) = \sum_{i=1}^N A_i^* X A_i
\]
for some complex $n \times n$ matrices $A_1,\ldots,A_N$, and such an operator is said to be \textbf{doubly stochastic} if
\[
    T(I_n) = I_n \quad \text{and} \quad T^*(I_n) = I_n,
\]
where
\[
    T^*(X) = \sum_{i=1}^N A_i X A_i^*.
\]
A CP operator is called \textbf{separable} (or entanglement-breaking) if the $A_i$ matrices are rank-one; i.e. $A_i = \bm{u}_i \bm{v}_i^*$ for some $\bm{u}_i,\bm{v}_i \in \C^n$ for all $i \in [N]$.

The \textbf{quantum permanent (QP)} is then an important quantity one can compute from a given CP operator. It is given by
\[
    \mathrm{QP}(T) = \mathbb{E}\left[\left|\det\left(\sum_{i=1}^N z_i A_i\right)\right|^2\right] = \left\langle \det\left(\sum_{i=1}^N x_i A_i\right), \det\left(\sum_{i=1}^N x_i A_i\right) \right\rangle,
\]
where the expectation is over i.i.d. standard complex Gaussians $z_1,\ldots,z_N$ and the inner product is the unique (up to scalar) unitarily-invariant inner product on $N$-variate homogeneous polynomials of degree $n$. For example, $\langle p, q \rangle = p(\partial_{\bm{x}}) q(\bm{x})$ has the correct normalization.

We now describe an alternative equivalent approach for separable doubly stochastic CP operators. For such an operator $T$, we have a (not necessarily unique) representation via
\[
    T(X) = \sum_{i=1}^N a_i \bm{v}_i \bm{u}_i^* X \bm{u}_i \bm{v}_i^*,
\]
where $\bm{u}_i,\bm{v}_i \in \C^n$ are unit vectors and $0 < a_i \leq 1$ for all $i \in [N]$, and
\[
    \sum_{i=1}^N a_i \bm{u}_i \bm{u}_i^* = \sum_{i=1}^N a_i \bm{v}_i \bm{v}_i^* = I_n \quad \text{and} \quad \sum_{i=1}^N a_i = n.
\]
In this case, the quantum permanent can be written as
\[
    \mathrm{QP}(T) = \langle p, q \rangle,
\]
where the inner product is the coefficient-wise dot product and
\[
    p(\bm{x}) = \det\left(\sum_{i=1}^N x_i \sqrt{a_i} \bm{u}_i \bm{u}_i^*\right) \quad \text{and} \quad q(\bm{x}) = \det\left(\sum_{i=1}^N x_i \sqrt{a_i} \bm{v}_i \bm{v}_i^*\right).
\]
This is precisely the quantity we want to lower-bound in \Cref{QP}. We note that in fact this coefficient-wise dot product is also the unitarily-invariant inner product.

\subsection{Bounding Separable Quantum Permanents}

With this, we can use results of \cite{AOG17} to prove a lower bound on the quantum permanent of $T$ in this separable doubly stochastic case. Noting that the polynomials in the inner product are multiaffine, we apply Theorem~1.1 of \cite{AOG17} to obtain
\[
    \mathrm{QP}(T) \geq (\bm{1}-\bm{a})^{\bm{1}-\bm{a}} \inf_{\bm{x},\bm{y} > \bm{0}} \frac{p(\bm{x}) q(\bm{y})}{(\bm{x}\bm{y}/\bm{a})^{\bm{a}}},
\]
where $\bm{a} \bm{b} := \prod_{i=1}^N a_i b_i$ and $\bm{a}^{\bm{b}} := \prod_{i=1}^N a_i^{b_i}$.

To lower-bound the infimum, note that $\log p(e^{\bm{y}})$ and $\log q(e^{\bm{y}})$ are both convex on $\R^n$ since $p$ and $q$ have non-negative coefficients (by H\"older's inequality). Note then that for $\bm{y} = \frac{1}{2} \log \bm{a}$ we have
\[
    \partial_{y_i} \log p(e^{\bm{y}}) = e^{y_i} \sqrt{a_i} \bm{u}_i^*\left(\sum_{i=1}^N e^{y_i} \sqrt{a_i} \bm{u}_i \bm{u}_i^*\right)^{-1} \bm{u}_i = a_i
\]
for all $i \in [N]$. Thus $\log p(e^{\bm{y}}) - \bm{y} \cdot \bm{a}$ is convex with zero gradient at $\bm{y} = \frac{1}{2} \log \bm{a}$, which implies
\[
    \frac{p(\bm{x})}{\bm{x}^{\bm{a}}} \geq \frac{p(\sqrt{\bm{a}})}{(\sqrt{\bm{a}})^{\bm{a}}}.
\]
The same thing holds for $q$, and since $p(\sqrt{\bm{a}}) = q(\sqrt{\bm{a}}) = 1$, we therefore have
\[
    \frac{p(\bm{x}) q(\bm{y})}{(\bm{x}\bm{y}/\bm{a})^{\bm{a}}} \geq \frac{p(\sqrt{\bm{a}}) q(\sqrt{\bm{a}})}{(\bm{a}/\bm{a})^{\bm{a}}} = 1
\]
for all $\bm{x} > \bm{0}$.

Applying this to the quantum permanent bound given above, we obtain
\begin{equation} \label{QP-bound}
    \mathrm{QP}(T) \geq (\bm{1}-\bm{a})^{\bm{1}-\bm{a}} \geq \left(1 - \frac{n}{N}\right)^{N-n} > e^{-n},
\end{equation}
where the second inequality follows from the fact that $\prod_{i=1}^N (1-t_i)^{1-t_i}$ is log-convex on $\{\bm{t} \in [0,1]^N : \bm{1}^\top \bm{t} = n\}$. For a stronger bound, Carath\'eodory's theorem says $N \leq n^4$, giving a lower bound of
\[
    \mathrm{QP}(T) \geq \left(1 - \frac{1}{n^3}\right)^{n(n^3-1)}.
\]
Note that we can remove the usual $+1$ from Carath\'eodory's theorem because of the affine double stochasticity constraint.

The above discussion then leads to the following.

\begin{theorem} \label{main-QP}
    If $T: \R^{n \times n} \to \R^{n \times n}$ is a separable completely positive operator, we have
    \[
        \cpc(T) \geq \mathrm{QP}(T) \geq \left(1 - \frac{1}{n^3}\right)^{n(n^3-1)} \cpc(T) > e^{-n} \cpc(T),
    \]
    where $\cpc(T) = \inf_{X \succ 0} \frac{\det(T(X))}{\det(X)}$.
\end{theorem}

The upper bound on $\mathrm{QP}(T)$ given in \Cref{main-QP} was proved in \cite{Gurvits2004}. We note that even this bound is not valid for general CP operators; the identity map already shows that at best we have $n! \cpc(T) \geq \mathrm{QP}(T)$, and in fact this holds in general.

\subsection{Mixed Discriminant Bounds from the Quantum Permanent}

We note that \Cref{main-QP} implies an interesting corollary on the mixed discriminant of doubly stochastic tuples of matrices. Fix $n \times n$ PSD matrices $A_1,\ldots,A_n$, given by
\[
    A_i = \sum_{j=1}^n \lambda_{ij} \bm{v}_{ij} \bm{v}_{ij}^*
\]
for all $i \in [n]$, where $\bm{v}_ij \in \C^n$ are unit vectors for all $i,j$. Then define
\[
    T(X) = \sum_{i=1}^n \sum_{j=1}^n \lambda_{ij} \bm{v}_{ij} \bm{e}_i^* X \bm{e}_i \bm{v}_{ij}^*.
\]
For all $k \in [n]$, this implies
\[
    T(\bm{e}_k \bm{e}_k^*) = \sum_{j=1}^n \lambda_{kj} \bm{v}_{kj} \bm{v}_{kj}^* = A_k,
\]
and $T(\bm{e}_a \bm{e}_b^*) = 0$ whenever $a \neq b$. Thus if $(A_1,\ldots,A_n)$ is doubly stochastic (i.e., $\tr(A_i) = 1$ for all $i \in [n]$ and $\sum_{i=1}^n A_i = I_n$) then we have
\[
    n! \MD(A_1,\ldots,A_n) = \mathrm{QP}(T) \geq \prod_{i,j=1}^n (1-\lambda_{ij})^{1-\lambda_{ij}},
\]
where the equality follows from Eq.~(15) of (the arXiv version of) \cite{Gurvits2004} and the inequality follows from \eqref{QP-bound}. (The $n!$ comes from the usual normalization of the mixed discriminant, which we adopted in this paper.)

This mixed discriminant bound is a generalization (with simpler proof) of the permanent inequality of Gurvits which was used to settle the asymptotic Friedland's lower matching conjecture in \cite{Gurvits2011Bethe,GurvitsSamorodnitsky2014}. We note that the Bethe approximation used for the permanent in those works is not applicable to the mixed discriminant.

Majorization of the diagonal entries by the eigenvalues then gives another lower bound on the mixed discriminant. Since $\prod_{i,j=1}^n (1-t_{ij})^{1-t_{ij}}$ is log-convex, it is maximized over any convex set at the extreme points. Thus in the above lower bound is maximized over all choices of $t_{ij}$ in the permutohedron defined by $(\lambda_{ij})_{i,j=1}^n$. Since the sequence of diagonal entries of the matrices $A_1,\ldots,A_n$ lies in this permutohedron by majorization, we thus further have
\[
    n! \MD(A_1,\ldots,A_n) = \mathrm{QP}(T) \geq \prod_{i,j=1}^n (1-a_{ij})^{1-a_{ij}},
\]
where $a_{ij} = (A_i)_{jj}$ is the $j$th diagonal entry of $A_i$ for all $i,j \in [n]$.

\subsection{Doubly Stochastic PSD Operators}

As a final note, we obtain a van der Waerden type bound for a the class of doubly stochastic PSD operators, with respect to the Frobenius (trace) inner product. That is, we will show that $\mathrm{QP}(T) \geq \frac{n!}{n^n}$ for all such operators $T$.

Note that an operator $T$ is doubly stochastic and PSD whenever $T(I) = I$ and there exists a PSD operator $S$ such that $T = S^2$. Since $T,S$ are Hermitian, we can consider their spectral decompositions:
\[
    T(X) = \sum_{i=1}^N \lambda_i V_i \langle V_i, X \rangle \quad \text{and} \quad S(X) = \sum_{i=1}^N \sqrt{\lambda_i} V_i \langle V_i, X \rangle,
\]
where $\lambda_i \geq 0$ and $\langle V_i, V_j \rangle = 0$ for all $i \neq j$. Since $\tr(X) = \langle I, X \rangle = \langle T(I), X \rangle = \langle I, T(X) \rangle = \tr(T(X))$, we thus have $\lambda_i = 1$ for all $V_i$ with non-zero trace. Therefore $T(I) = S(I)$, and thus $S$ is also doubly stochastic.

We now bound $\mathrm{QP}(T)$. Another formulation of the quantum permanent is given by
\[
    \mathrm{QP}(T) = \langle \det(T(X)), \det(X) \rangle,
\]
where the inner product is given by $\langle P, Q \rangle = P(\partial_X) Q(X)$. Note that this normalization implies $\langle \det(X), \det(X) \rangle = n!$ as expected. Now note that
\[
    \langle (\tr(X))^n, P(X) \rangle = \nabla_I^n P(X) = n! P(I)
\]
for all $n$-homogeneous $P$. Thus
\[
    \langle (\tr(X))^n, \det(S(X)) \rangle = n! \quad \text{and} \quad \langle (\tr(X))^n, (\tr(X))^n \rangle = n! n^n.
\]
Since $S = S^*$, unitary invariance of the inner product implies
\[
    \mathrm{QP}(T) = \langle \det(S^2(X)), \det(X) \rangle = \langle \det(S(X)), \det(S(X)) \rangle,
\]
and by Cauchy-Schwarz we have
\[
    n! = \langle (\tr(X))^n, \det(S(X)) \rangle \leq \|(\tr(X))^n\| \cdot \|\det(S(X))\| = \sqrt{n! n^n \mathrm{QP}(T)}.
\]
Rearranging then implies the desired bound $\mathrm{QP}(T) \geq \frac{n!}{n^n}$.

\printbibliography

\end{document}